\documentclass[10pt,a4paper,reqno]{amsart}
\usepackage{amsmath,amssymb}
\usepackage[margin=1in]{geometry}
\usepackage[colorlinks=true,citecolor=magenta,linkcolor=blue]{hyperref}
\usepackage{color}
\newtheorem{theorem}{Theorem}[section] 
\newtheorem{lemma}[theorem]{Lemma}     

\newtheorem{proposition}[theorem]{Proposition}
\newtheorem{definition}{Definition}[section]
\newtheorem{remark}{Remark}

\newcommand{\C}{\mathcal{C}}

\newcommand{\yy}{\boldsymbol{y}}
\newcommand{\mm}{\boldsymbol{\mu}}

\newcommand{\WW}{\mathbf{C}_{\infty}}
\newcommand{\cc}{\boldsymbol{c}}
\newcommand{\uu}{\boldsymbol{u}}
\newcommand{\ww}{\boldsymbol{c}_{\infty}}
\newcommand{\CC}{\mathbf{C}}

\usepackage{tikz}
\usetikzlibrary{arrows,automata,backgrounds,calendar,mindmap,trees}
\usepackage{xcolor}

\title[Convergence to equilibrium for reaction-diffusion systems] 
 {Convergence to equilibrium of renormalised solutions to nonlinear chemical reaction-diffusion systems} 

\author{Klemens Fellner and Bao Quoc Tang}
\address{Klemens Fellner \hfill\break
Institute of Mathematics and Scientific Computing, University of Graz, Heinrichstrasse 36, 8010 Graz, Austria}
\email{klemens.fellner@uni-graz.at}
\address{Bao Quoc Tang \hfill\break
Institute of Mathematics and Scientific Computing, University of Graz, Heinrichstrasse 36, 8010 Graz, Austria}
\email{quoc.tang@uni-graz.at} 

\begin{document}

\begin{abstract}
The convergence to equilibrium for renormalised solutions to nonlinear reaction-diffusion systems is studied. The considered reaction-diffusion systems arise from chemical reaction networks with mass action kinetics and satisfy the complex balanced condition. By applying the so-called entropy method, we show that if the system does not have boundary equilibria, then any renormalised solution converges exponentially to the complex balanced equilibrium with a rate, which can be computed explicitly up to a finite dimensional inequality. This inequality is proven via a contradiction argument and thus not explicitly. 
An explicit method of proof, however, is provided 
for a specific application modelling a reversible enzyme reaction by exploiting the specific structure of the conservation laws.

Our approach is also useful to study the trend to equilibrium for systems possessing boundary equilibria. More precisely, to show the convergence to equilibrium for systems with boundary equilibria, we establish a sufficient condition in terms of a modified finite dimensional inequality along trajectories of the system. By assuming this condition, which roughly means that the system produces too much entropy to stay close to a boundary equilibrium for infinite time, 
the entropy method  
shows exponential convergence to equilibrium for renormalised solutions to complex balanced systems with boundary equilibria.
\end{abstract}
\maketitle
\tableofcontents

\section{Introduction and Main results}

The large time behaviour of reaction-diffusion systems is a long standing and yet highly active topic in the analysis of partial differential equations. Classical methods include dynamical systems, invariant regions or linearisation methods. Recently, the so-called entropy method, a fully nonlinear approach, proved to be very useful in studying convergence to equilibrium for many PDE systems, in particular reaction-diffusion systems which feature a suitable dissipative structure. 

\medskip
The aim of the present paper is to prove convergence to equilibrium for nonlinear reaction-diffusion systems arising from chemical reaction networks. A chemical reaction network is a quadruple $\{\mathcal S, \mathcal C, \mathcal R, \mathcal K\}$ in which $\mathcal S = \{S_1, \ldots, S_N\}$ denotes the set of chemical substances, $\mathcal C = \{\yy_1, \ldots, \yy_{|\mathcal C|}\}$ with $\yy_i \in {(\{0\}\cup[1,\infty))}^N$, $i=1,\ldots, |\mathcal C|$ is the set of chemical  complexes, which are either reactants and/or products of a chemical reaction and $\yy=(y_j)_{j=1}^N$ denotes a vector of stoichiometric coefficients for the substances $S_1,\ldots, S_N$, where $y_j\not = 0$ iff the substance $S_j$ is part of the complex $\yy$. Correspondingly, $\mathcal R$ is the set of all considered chemical reactions $\yy_r \to \yy_r'$ with $\yy_r, \yy_r' \in \mathcal C$ and for $r=1,\ldots, |\mathcal R|$. Moreover, $\mathcal K = \{k_r: r=1,\ldots, |\mathcal R| \}$ is set of the associated reaction rate constants with $k_r>0$ being the rate of the reaction $\yy_r \to \yy_r'$ for all $r=1,\ldots, |\mathcal R|$. The reaction network $\{\mathcal S, \mathcal C, \mathcal R, \mathcal K\}$ is assumed to satisfy the following natural conditions:
\begin{itemize}
	\item[(1)] for each $S_i\in\mathcal S$, there exists at least one complex $\yy\in \mathcal C$ for which {the corresponding stoichiometric coefficient $y_i$ is nontrivial, i.e.} $y_i\geq 1$,
	\item[(2)] there exist no trivial reaction $\yy \to\yy \in \mathcal R$ for any complex $\yy\in \mathcal C$,
	\item[(3)] for any $\yy\in \mathcal C$, there must exist a $\yy'\in\mathcal C$ such that $\yy \to \yy'\in \mathcal R$ or $\yy'\to\yy \in \mathcal R$, {i.e. every complex must be either reactant or product of at least one reaction.}
\end{itemize}
{In the following, we shall use the convention that the primed complexes $\yy_r'\in\mathcal C$ (respectively $\yy'\in\mathcal C$) denote the product of the $r$-th reaction while the unprimed complexes $\yy_r\in\mathcal C$ (respectively $\yy \in \mathcal C$) denote the reactant; except when specified otherwise.}

As an example for a reaction network, we find for the single reversible reaction 
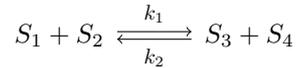
\begin{equation*}
	\begin{tikzpicture} [baseline=(current  bounding  box.center)]
	\node (a) {$S_1 + S_2$} node (b) at (2.5,0) {$S_3+S_4$};
	\draw[arrows=->] ([xshift =0.5mm, yshift=0.5mm]a.east) -- node [above] {\scalebox{.8}[.8]{$k_1$}} ([xshift =-0.5mm,yshift=0.5mm]b.west);
	\draw[arrows=->] ([xshift =-0.5mm, yshift=-0.5mm]b.west) -- node [below] {\scalebox{.8}[.8]{$k_2$}} ([xshift =0.5mm,yshift=-0.5mm]a.east);
	\end{tikzpicture}
	\end{equation*}
the set of chemical substances $\mathcal S = \{S_1, S_2, S_3, S_4\}$, the set of complexes $\mathcal C = \{(1,1,0,0), (0,0,1,1)\}$ comprising the two complexes $S_1 + S_2$ and $S_3 + S_4$, the set of two reactions $\mathcal R = \{(1,1,0,0)\to(0,0,1,1), \; (0,0,1,1)\to (1,1,0,0)\}$ and the set of the corresponding reaction rate constants $\mathcal K = \{k_1, k_2\}$.

\medskip
To specify a reaction-diffusion system modelling a chemical reaction network $\{\mathcal S, \mathcal C, \mathcal R, \mathcal K\}$, we assume the reactions to take place in a bounded vessel (or reactor) $\Omega\subset \mathbb R^n$, where $\Omega$ is a bounded domain with Lipschitz boundary. We also assume (w.l.o.g. after a suitable rescaling of the space variable) that $\Omega$ has normalised volume, i.e. 
$$
|\Omega| = 1.
$$
We denote by $\cc(x,t) = (c_1(x,t), \ldots, c_N(x,t))$ the vector of concentrations where $c_i(x,t)$ is the concentration of $S_i$ at time $t>0$ and position $x\in\Omega$. Each substance $S_i$ is assumed to diffuse in $\Omega$ with a strictly positive diffusion coefficient $d_i>0$. 
The corresponding reaction-diffusion system then reads as
\begin{equation}\label{e0}
\frac{\partial}{\partial t}\cc - \mathbb D\Delta \cc = \mathbf{R}(\cc) \quad \text{ for } \quad (x,t)\in \Omega\times (0,+\infty),
\end{equation}
where the diffusion matrix $\mathbb D = \mathrm{diag}(d_1, \ldots, d_N)$ is %assumed to be %uniformly 
positive definite since $d_i>0$ for all $i=1,\ldots,N$, and $\mathbf{R}(\cc)$ represents all the reactions in $\mathcal R$. We shall apply the {\it law of mass action} to get an explicit form of $\mathbf{R}(\cc)$ which reads as
%by the mass action law, i.e. 
\begin{equation}\label{Reaction}
\mathbf{R}(\cc) = \sum_{r=1}^{|\mathcal R|}k_r\cc^{\yy_r}(\yy_r' - \yy_r) \quad \text{ with } \quad \cc^{\yy_r} = \prod_{i=1}^{N}c_i^{y_{r,i}},
\end{equation}
where $k_r>0$ denotes the reaction rate constant of the $r$-th reaction. Finally, system \eqref{e0} is subject to
nonnegative initial data ${\bf{c}}_0(x) \ge 0$ (by which we mean ${\bf{c}}_0(x) := (c_{1,0}(x),..,c_{N,0}(x))$ and $c_{i,0}(x) \ge 0$ for $i=1,..,N$ and $x\in\Omega$), and homogeneous Neumann boundary conditions 
\begin{equation}\label{Nin}
\cc(0,x) = {\bf{c}}_0(x) \,\, {\hbox{ for }}\,\,  x\in \Omega ,
\quad\text{ and } \quad 
\nabla \cc \cdot \nu = 0 \,\,{\hbox{ for }}\,\, (x,t)\in \partial\Omega\times \mathbb{R}_+,
\end{equation} 
where $\nu := \nu(x)$ is the outward normal unit vector at point $x \in \partial\Omega$.
\medskip 

Many chemical reaction networks exhibit mass conservation laws. For system \eqref{e0}--\eqref{Nin} we define the Wegscheider matrix $W = [(\yy_r' - \yy_r)_{r=1,\ldots,|\mathcal R|}]^{\top}\in \mathbb R^{|\mathcal R|\times N}$ and denote by $m = \mathrm{codim}(W)$. Then, if $m>0$, there exists a (non-unique) matrix $\mathbb Q\in \mathbb R^{m\times N}$ whose rows are formed by linear independent (left-zero) eigenvectors of $W^{\top}$. It follows from \eqref{Reaction} that $\mathbf R(\cc) \in \mathrm{range}(W^{\top})$ and thus
\begin{equation}\label{11ter}
\mathbb Q\, \mathbf{R}(\cc) = 0 \quad \text{ for all } \quad \cc\in \mathbb R^N_{+}.
\end{equation}
Therefore it follows from \eqref{e0} that $\partial_t (\mathbb Q\,\cc) - \mathbb Q\mathbb D \Delta \cc = \mathbb Q\,\mathbf{R}(\cc) = 0$ and hence due to the homogeneous Neumann boundary condition, the co-dimension of the Wegscheider's matrix $W$ leads to $m$ {\it (linearly independent) mass conservation laws} of the following form 
\begin{equation}\label{12ter}
\frac{d}{dt}\mathbb Q\,\overline{\cc}(t) = 0 \Longrightarrow \mathbb Q\,\overline{\cc}(t) = \mathbb Q\,\overline{\cc_0} =: M \in \mathbb R^m \quad \text{ for all } \quad t>0,
\end{equation}
where $\overline{\cc} = (\overline{c_1}, \ldots, \overline{c_N})$ and $\overline{c_i} = \int_{\Omega}c_i dx$ (after recalling that $|\Omega| = 1$), and $M$ is called an {\it initial mass vector}, which depends on the choice of $\mathbb Q$. By changing the signs of some rows of $\mathbb Q$ if necessary, we can always consider (w.l.o.g.) a matrix $\mathbb Q$ such that the initial mass vector $M$ is non-negative, i.e. $M \in \mathbb R_{+}^m$.

\medskip
To state the main results of this paper, we need the following definitions concerning {\it equilibria} of chemical reaction networks.
% under study.

\begin{definition}[Equilibria]\label{Equilibria} \hfill\\
Consider a chemical reaction network $\{\mathcal S, \mathcal C, \mathcal R, \mathcal K\}$ which is modelled by the reaction-diffusion system \eqref{e0}--\eqref{Nin}. Denote by $M = \mathbb Q\,\overline{\cc_0}$ the initial mass vector. Let $\ww := (c_{1,\infty},..,c_{N,\infty})\in \mathbb{R}^N_{+}$ satisfy the mass conservation laws $\mathbb Q\,\ww = M$. Then,
	\begin{itemize}
		\item $\ww$ is called an {\normalfont{equilibrium}}  if $\mathbf{R}(\ww) = 0$.
		
		\item $\ww$ is called a {\normalfont{detailed balanced equilibrium}}  if for each forward reaction $\yy \xrightarrow{k_{f}} \yy'$ (with $k_f>0$) in $\mathcal R$, there exists in $\mathcal{R}$ also the corresponding backward reaction $\yy' \xrightarrow{k_{b}} \yy$ (with $k_b>0$) and 
\begin{equation*}
		k_f\ww^{\yy} = k_b\ww^{\yy'}.
\end{equation*}
		
\item 	$\ww$ is called a {\normalfont{complex balanced equilibrium}} if the total outflow %(left hand side) 
and inflow %(right hand side) 
at the equilibrium $\ww$ are equal for every complex $\yy\in \C$, i.e. for any complex $\yy\in \C$, we have
\begin{equation} \label{ComplexBalance}
(\text{total outflow from } \yy)\quad = \sum_{\{r:\, \yy_r = \yy\}}k_r\ww^{\yy_r} = \sum_{\{s:\, \yy_s' = \yy\}}k_s\ww^{\yy_s} \quad =\quad (\text{total inflow into } \yy),
\end{equation}
where $\{s:\, \yy_s' = \yy\}$ denotes the set of all reactions 
$\yy_s \xrightarrow{k_{s}>0} \yy_s'$ with fixed product complex 
$\yy_s'=\yy\in \C$.
		
\item $\ww$ is called a {\normalfont{boundary detailed/complex balanced equilibrium}} (or shortly a {\normalfont{boundary equilibrium}}) if $\ww$ is a detailed/complex balanced equilibrium and $\ww\in \partial\mathbb{R}^N_{+}$.
\item 
{
It follows directly from the above definitions that
\begin{equation*}
\begin{gathered}
\ww \text{ is a detailed balanced equilibrium} \Longrightarrow \ww \text{ is a complex balanced equilibrium}\\
\Longrightarrow \ww \text{ is an equilibrium},
\end{gathered}
\end{equation*}
but the reverse is in general not true. }
	\end{itemize}		

A chemical reaction network is called  {\it complex balanced} if for each strictly positive mass vector $M \in \mathbb R^m_{>0}$ it possesses a strictly positive (i.e. not a boundary) complex balanced equilibrium.
\end{definition}
\medskip

The concept of detailed balance goes back as far as Boltzmann for modelling collisions in kinetic gas theory and for proving the H-theorem for Boltzmann's equation \cite{Bol1896}. It was then applied to chemical kinetics by Wegscheider \cite{Weg1901}. The complex balanced condition was also considered by Boltzmann \cite{Bol1887} under the name \emph{semi-detailed balanced condition} or \emph{cyclic balanced condition}, and was 
systematically used by Horn, Jackson and Feinberg in the seventies for chemical reaction network theory, see e.g. \cite{FeHo74,Hor72,HoJa72}.

It is well-known that if a chemical reaction network (such as modelled by system \eqref{e0}-\eqref{Nin}) has one complex balanced equilibrium,
then all other possible equilibria (independently of the initial mass vector) are necessarily also complex balanced, see e.g. \cite{Hor72,HoJa72}. 
Moreover, for every fixed positive initial mass vector $M\in \mathbb R^m_{>0}$, there exists a unique complex balanced equilibrium $\ww\in \mathbb R^N_{>0}$ satisfying the mass conservation laws determined by the initial mass vector $M\in \mathbb R^m_{>0}$. Note that (possibly infinitely) many boundary equilibria may exist as well. 

Throughout this paper, 
we will refer to this strictly positive equilibrium as \emph{the complex balanced equilibrium} while all other equilibria are simply called boundary equilibria. 
Moreover, we will consider 
positive initial mass vectors $M\in \mathbb R^m_{>0}$ in order to ensure that any considered complex balanced network 
features a positive complex balanced equilibrium $\ww\in \mathbb R^N_{>0}$. Note that  all our results hold equally 
true for  
non-negative initial mass vectors $M\in \mathbb R^m_{+}$ as long as there exists a unique positive complex balanced equilibrium $\ww\in \mathbb R^N_{>0}$, which will typically (but not always) be the case.
\medskip

{This paper aims to prove exponential convergence to equilibrium of solutions to the nonlinear reaction-diffusion system \eqref{e0}-\eqref{Nin}  
under the assumption the considered chemical reaction network is complex balanced.}

{
The method of proof is the so-called \emph{entropy method}. 
The main idea of the entropy method is to qualitatively exploit the decay of a  suitable entropy (e.g. convex Lyapunov) functional $E[f]$ along a trajectory $f$ of an evolution process: }
\begin{equation*}
	-\frac{d}{dt}E[f] = D[f] \geq 0,%\label{EDissstrong}
\end{equation*}
{where $D[f]$ is called entropy production functional or also  entropy dissipation functional in cases when $E[f]$ is physically an energy functional. 
The latter is the case for nonlinear complex balanced reaction-diffusion systems of the form \eqref{e0}--\eqref{Nin}, 
where the following logarithmic relative free energy functional 
\begin{equation}\label{FreeEnergy_PDE}
\qquad\quad\mathcal{E}(\cc|\ww) = \sum_{i=1}^{N}\int_{\Omega}\left(c_i\log\frac{c_i}{c_{i,\infty}} - c_i + c_{i,\infty}\right)dx, \hfill \qquad\quad \text{(PDE entropy functional)}
\end{equation}
constitutes a suitable entropy functional. Note that for nonlinear 
reaction-diffusion systems, the above logarithmic relative entropy is the only generally existing Lyapunov functional, while for linear complex balanced systems, other generalised relative entropy functional do also exist, see e.g. \cite{FPT15,Per07}.}
 The following explicit form of the entropy dissipation functional $\mathcal D(\cc)$
associated to \eqref{FreeEnergy_PDE} along the flow of 
system \eqref{e0}--\eqref{Nin} was derived in \cite{DFT16}: 
\begin{equation}\label{ED_PDE}
-\frac{d}{dt} \mathcal{E}(\cc|\ww) = \mathcal{D}(\cc) := \sum_{i=1}^{N}d_i \int_{\Omega}\frac{|\nabla c_i|^2}{c_i}dx + \sum_{r=1}^{|\mathcal R|}k_r\,\ww^{\yy_r}\int_{\Omega}\Psi\left(\frac{\cc^{\yy_r}}{\ww^{\yy_r}}; \frac{\cc^{\yy_r'}}{\ww^{\yy_r'}}\right)dx \geq 0,
\end{equation}
where $\Psi: [0,+\infty)\times [0,+\infty) \rightarrow \mathbb R_+\cup\{+\infty\}$ is defined by
$\Psi(x,y) =  x\log(x/y) - x + y\ge0$.
\medskip

The entropy method applies to general evolution processes, which are well behaved in the sense that 
\[
	D[f] = 0\quad \text{ and \quad $f$ satisfies all conservation laws} \qquad \Longleftrightarrow \qquad f= f_{\infty}.
\]
This condition holds true for the system
\eqref{e0}--\eqref{Nin}, where $\mathcal{D}(\cc)=0$ is satisfied by all constant states which balance the 
reactions of the complex balanced network. Thus, provided no boundary equilibria exist, taking into account all conservation laws uniquely identifies the complex balanced equilibrium $\ww$.

Given such a well-behaved evolution process, the entropy method aims to quantify the decay of the 
entropy functional $E[f]$ in terms of the relative entropy towards the equilibrium state. 
More precisely, the goal is 
an {\it entropy-entropy production estimate} 
(which is a functional inequality independent of the 
flow of the evolution process) 
of the form
\[
	D[f] \geq \Phi(E[f] - E[f_{\infty}]),
\]
where $\Phi(x)\geq 0$ and $\Phi(x) = 0\Leftrightarrow x = 0$. 
More specifically for system
\eqref{e0}--\eqref{Nin}, the first key result of this paper  is to prove the following {entropy-entropy dissipation estimate}
(or rather free energy-free energy dissipation estimate)
\begin{equation}\label{MainEstimate}
	\mathcal{D}(\cc) \geq \lambda\, \mathcal{E}(\cc|\ww)
\end{equation}
for some constant $\lambda >0$. 

Assuming that such a functional inequality is proven and that a suitable concept of solutions 
to system \eqref{e0}--\eqref{Nin} satisfyies a weak entropy-entropy dissipation law (i.e. an integrated version of the formal relation \eqref{ED_PDE}) of the form
\begin{equation}\label{weak}
	\mathcal E(\cc(t)|\ww) + \int_{s}^{t}\mathcal D(\cc(\tau))d\tau \leq \mathcal{E}(\cc(s)|\ww)\qquad \text{for almost all}\quad 0\leq s< t, 
\end{equation}
then a Gronwall argument implies exponential convergence to equilibrium first in relative entropy, i.e.
\begin{equation*}
\mathcal E(\cc(t)|\ww) \le \mathcal{E}(\cc(0)|\ww) \,e^{-\lambda t}
\end{equation*}
and consequently in $L^1$-norm, thanks to a Csisz\'{a}r-Kullback-Pinsker type inequality,
see Lemma \ref{lem:CKP} below. 
\medskip

In the first main results of this paper, we prove for general, complex balanced reaction-diffusion systems \eqref{e0}--\eqref{Nin} \emph{without} boundary equilibria, 
that any so-called renormalised solution (which is the only existing solution concept for such a general class of nonlinear reaction-diffusion systems, see Theorem \ref{renormalised}) converges exponentially to the complex balanced equilibrium with a rate which can be explicitly estimated in terms of {the systems' parameters and a constant obtained from a finite dimensional inequality with mass conservation constraints.} More precisely, our first main theorem reads as
\begin{theorem}[Convergence to equilibrium for general complex balanced reaction-diffusion systems without boundary equilibria]\label{theo:main}
\hfill\\
Let $\Omega$ be a bounded domain in $\mathbb R^n$ with Lipschitz boundary $\partial\Omega$. Assume that the diffusion matrix $\mathbb D$ is positive definite, i.e. $d_i>0$ for all $i=1,\ldots, N$. Moreover, we assume that system \eqref{e0}--\eqref{Reaction} is complex balanced. Consequently, for each positive initial mass vector $M \in \mathbb R^m_{>0}$ there exists a unique positive complex balanced equilibrium $\ww\in \mathbb R^N_{>0}$.
Assume in addition that system \eqref{e0}--\eqref{Reaction} {\normalfont does not have boundary equilibria}. 

	\medskip
	Then, for all states $\overline{\cc}\in \mathbb R^N_{>0}$ satisfying $\mathcal{E}(\overline{\cc}|\ww)<+\infty$ and $\mathbb Q\,\overline{\cc} = M$, there 
exists a constant $H_1>0$ depending only on $\mathbb Q$, the stoichiometric coefficients $\yy\in \mathcal C$, $M$ and $\mathcal{E}(\overline{\cc}|\ww)$
such that 	
\begin{equation}\label{last_1_re}
		\sum_{r=1}^{|\mathcal R|}\left[\sqrt{\frac{\overline{\cc}}{\ww}}^{\yy_r} - \sqrt{\frac{\overline{\cc}}{\ww}}^{\yy_r'}\right]^2 \geq H_1\sum_{i=1}^{N}\left(\sqrt{\frac{\overline{c_i}}{c_{i,\infty}}} - 1\right)^2.
	\end{equation}
	Here
	\[
		\sqrt{\frac{\overline{\cc}}{\ww}} = \left(\sqrt{\frac{\overline{c_1}}{c_{1,\infty}}}, \ldots, \sqrt{\frac{\overline{c_N}}{c_{N,\infty}}}\,\right).
	\]
	
	\medskip
Further, inequality \eqref{last_1_re} implies that  for all measurable vector functions $\cc: \Omega \to \mathbb R_+^{N}$ satisfying $\mathcal{E}(\overline{\cc}|\ww)\leq K$ and $\mathbb Q\,\overline{\cc} = M$, the entropy-entropy dissipation inequality
\eqref{MainEstimate}, i.e.  
\begin{equation*}%\label{eq:entropy}
\mathcal{D}(\cc) \geq \lambda\, \mathcal{E}(\cc|\ww)
\end{equation*}
holds with $\lambda = \frac 12 \min\{\lambda_1,K_2H_1/K_1\}$, where $\lambda_1 = C_{\mathrm{LSI}}\min_{i=1,\ldots, N}\{d_i\}$ and $C_{\mathrm{LSI}}$ is the constant in the Logarithmic Sobolev inequality (see Lemma \ref{lem:2}), and $K_1$ and $K_2$ are constants (see \eqref{K1} and \eqref{K2}) {\normalfont{depending explicitly}} on the domain $\Omega$, the diffusion matrix $\mathbb D$, the stoichiometric coefficients $\yy\in \mathcal C$, the reaction rate constants $k_r$, the initial mass $M$, the complex balanced equilibrium $\ww$ and the constant $K$. 
	
	\medskip
	Finally, as a consequence of functional inequality \eqref{MainEstimate}, any renormalised solution $\cc(x,t)$ to \eqref{e0}--\eqref{Nin} (see Theorem \ref{renormalised}) associated with initial data $\cc_0$ satisfying $\mathbb{Q}\,\overline{\cc_0} = M$ and $\mathcal E(\cc_0|\ww)<+\infty$, converges exponentially to $\ww$ in $L^1$-norm with the rate $\lambda/2$, that is
	\begin{equation}\label{convergence}
		\sum_{i=1}^{N}\|c_i(t) - c_{i,\infty}\|_{L^1(\Omega)}^2 \leq C_{\mathrm{CKP}}^{-1}\,\mathcal{E}(\cc_0|\ww)\,e^{-\lambda t} \qquad \text{ for almost all }\ t>0,
	\end{equation}
	where $C_{\mathrm{CKP}}$ is the constant in a Csisz\'ar-Kullback-Pinsker type inequality (see Lemma \ref{lem:CKP}).
\end{theorem}

\begin{remark}
Note that while an explicit bound for $H_1$ in \eqref{last_1_re}
can certainly be obtained near the equilibrium $\ww$ via Taylor expansion, such bounds far from equilibrium are highly nontrivial and an open problem due to the non-convexity of the involved nonlinear terms. Moreover, an additionally difficulty stems from the lack of a constructive approach to characterise and exploit the matrix $\mathbb Q$. 
\end{remark}

Theorem \ref{theo:main} comprises in our opinion the most general 
equilibration result for complex balanced reaction-diffusion systems, which is currently feasible. It 
generalises previous results on the exponential convergence to equilibrium for 
reaction-diffusion systems, partially in terms of considering complex balanced instead of detailed balanced systems, 
partially in terms of applying to renormalised solutions rather than weak- or classical solutions, and partially 
that the obtained convergence rate $\lambda$ is explicitly stated in terms of the key constant $H_1$. 
%Note that an explicit bound for $H_1$ for general reaction networks is still an open problem.
\medskip

At this point, we review some previous results concerning the large time behaviour of 
reaction-diffusion systems arising from chemical reaction networks:
\medskip
\begin{itemize}
	\item[$\bullet$] The first results on the entropy methods for nonlinear reaction-diffusion systems trace back to works of Gr\"{oger}, Glitzky and H\"{u}nlich \cite{Gro83,Gro86,Gro92,GGH96,GH97}, where the authors consider electro-chemical drift-diffusion-recombination models. However, the proof of associated entropy-entropy dissipation estimate was based on a contradiction argument in combination with a compactness method, thus provided only convergence to equilibrium in space dimension two and without explicit control of the rate of convergence.
	\medskip
	\item[$\bullet$] The first quantitative results providing convergence to equilibrium with explicit constants were obtained in \cite{DeFe06,DeFe08}, which considered prototypical nonlinear reactions of the form $2S_1 \leftrightharpoons S_2$, $S_1 + S_2 \leftrightharpoons S_3$ or $S_1 + S_2 \leftrightharpoons S_3+S_4$. Various generalisations were treated in \cite{FL16,BaFeEv14,Jan16}. Note that all these works consider special cases of \eqref{e0}--\eqref{Nin}.
	\medskip 
\item[$\bullet$] {For detailed balanced systems without boundary equilibria, a first general approach to prove exponential convergence to equilibrium for \eqref{e0}--\eqref{Nin} was presented in \cite{MiHaMa14}. 
The inspired key idea of \cite{MiHaMa14} was to prove an entropy-entropy dissipation estimate via a suitable
convexification argument (of the non-convex sum of reaction terms in \eqref{ED_PDE}). The disadvantage, however, is that except in special cases (e.g. $2S_1 \leftrightharpoons S_2$) the convexification argument seems not to allow for 
explicit estimates on the rate of convergence. The results of \cite{MiHaMa14} were extended in \cite{DFT16} to complex balanced systems thanks to the derivation of the entropy dissipation \eqref{ED_PDE}.  }
%By assuming the detailed balance condition and no boundary equilibria, this work was able to show the exponential convergence to equilibrium for general system \eqref{e0}. However, the proofs in these works are based on a (rather elegant) convexification argument and thus prevent to estimate explicitly the rate of convergence, except in special cases, e.g. $2S_1 \leftrightharpoons S_2$ (see \cite{MiHaMa14}).
	\medskip
\item[$\bullet$] In a recent work \cite{FT15}, we proposed a constructive approach 
to show exponential convergence to equilibrium for general detailed balanced reaction-diffusion systems, 
which allows to obtain explicit bounds on the rates of convergence in contrast to the convexification argument of  \cite{MiHaMa14}. The applicability of the constructive approach was demonstrated for two typical example systems: i) a reversible reaction of arbitrary many chemical substances $\alpha_1 S_1 + \cdots + \alpha_I S_I \leftrightharpoons \beta_1B_1 + \cdots + \beta_J B_J$ ($\ast$) and ii) a reversible enzyme reaction $S_1 + S_2 \leftrightharpoons S_3 \leftrightharpoons S_4 + S_5$. This approach is also applicable to complex balanced systems as demonstrated in \cite{DFT16} for a cyclic reaction $\alpha_1 S_1 \to \alpha_2S_2 \to \cdots \to \alpha_NS_N \to \alpha_1 S_1$. Also in \cite{FT15}, we provided an If-Theorem that for any detailed balanced systems, under the assumption of a finite dimensional inequality (like \eqref{last_1_re}) and a {technical non-degeneracy assumption on the entropy dissipation}, 
%with concentrations being close to the boundary $\partial\mathbb R^N_+$
then the solutions converge exponentially to the positive equilibrium with explicit rates. 
{In this paper, we are able to remove these technical assumptions as well as generalise the result to complex balanced systems.}
It is also worth mentioning that the reversible reaction with arbitrary chemical substances ($\ast$) was also recently treated in the  paper \cite{PSZ}.
\end{itemize}
Altogether, these previous results prove either exponential convergence for general systems at the price of a lack of explicitness of convergence rates, or they showed explicit rates of convergence for some special classes of reaction-diffusion systems. 

\medskip
The results of Theorem \ref{theo:main} improve the previous results in several directions:
\begin{itemize}
	\item[i.] We prove the functional inequality \eqref{MainEstimate} explicitly up to the finite dimensional inequality \eqref{last_1_re}. More precisely, Theorem \ref{theo:main} states that the constant $\lambda$ in \eqref{MainEstimate} scales with the minimum of $\lambda_1$ (derived from the diffusion coefficients and the Logarithmic Sobolev Inequality) and the constant $H_1$ from \eqref{last_1_re} times the structural constant $K_2/K_1$ with $K_1$ and $K_2$ given in \eqref{K1} and \eqref{K2}. 
We note that the idea of proving \eqref{MainEstimate} by using a finite dimensional version was already considered in \cite{MiHaMa14}. However, the approach therein lacks explicitness due to the use of the convexification argument. 	
%	\medskip
	\item[ii.] We provide a general result of exponential convergence to equilibrium for complex balanced systems without boundary equilibria. In particular, the rate of convergence is explicitly controlled in terms the constant $H_1 $ of the finite dimensional inequality \eqref{last_1_re} (and other explicit parameters). It is emphasised that although the constant $H_1 $ is not explicit in general, we believe it is possible to explicitly estimate $H_1$ in any concrete system once the mass conservation laws are explicitly known (see Section \ref{Applications} for such a system arising from reversible enzyme reactions).
%	\medskip
	\item[iii.] {Another important advantage of Theorem \ref{theo:main} and our method of proof is its role in a  potential strategy to consider systems with boundary equilibria. This leads to the second main result of this paper, which is discussed in the following paragraphs.}
\end{itemize}

\medskip
{It is important to point out that the entropy-entropy dissipation inequality \eqref{MainEstimate}, and consequently the finite dimensional inequality \eqref{last_1_re}, cannot hold for general systems with boundary equilibria:
If a solution trajectory of such a system should approach a boundary equilibrium, then the entropy dissipation $\mathcal{D}(\cc)$ tends to zero while the relative entropy to the complex balanced equilibrium $\mathcal{E}(\cc|\ww)$ remains positive, see e.g. \cite{DFT16,FT15}
for the details.}
Consequently, an entropy-entropy dissipation estimate of the form \eqref{MainEstimate} 
cannot hold. 

This structural difficulty is already encountered in complex balanced reaction networks in the ODE setting, 
i.e. by considering the solution $\boldsymbol{u}(t)$, which satisfies the ODE system
\begin{equation}\label{e00}
	\frac{d}{d t}\boldsymbol{u} = \mathbf{R}(\boldsymbol{u}),
	\end{equation}
where $\mathbf R(\boldsymbol{u})$ is defined as \eqref{Reaction} with $\boldsymbol{u}$ in place of $\cc$. 
There is an extensive literature concerning the large time asymptotics of complex balanced systems of the form 
\eqref{e00}. Indeed, it is proven that the unique strictly positive complex balanced equilibrium of an ODE reaction network  is locally stable (cf. \cite{HoJa72}). 
Moreover, it is conjectured that the positive complex balanced equilibrium is in fact globally stable, i.e. it is the unique global attractor for the dynamical system given by the ODE network (with the exception of initial data starting on $\partial \mathbb{R}_{+}^{N}$).
This statement is usually called the {\it Global Attractor Conjecture} (GAC) and has remained one of the most important open problems in the theory of chemical reaction networks, see e.g. \cite{And11,CNN13,GMS14,Pan12} and the references therein.  
A recently proposed proof of this conjecture in the ODE setting is currently under verification \cite{Cra15}.
\medskip

For reaction-diffusion systems of the form \eqref{e0}--\eqref{Nin}, it was pointed out in \cite[Remark 3.6]{DFT16} that if the boundary equilibria are unstable in the sense that solution trajectories cannot stay too close to those equilibria (in $L^1$-norm distance) for too long, then the convergence to the complex balanced equilibrium follows via a contradiction argument. However, proving such an instability for boundary equilibria is usually a subtle issue, {in particular in the PDE setting \eqref{e0}--\eqref{Nin}.} 

In this paper, by {using elements of the proof of Theorem \ref{theo:main}, we establish a weaker condition entailing instability of boundary equilibria and convergence to the complex balanced equilibrium}. More precisely, our condition {is based on a quantitative estimate that solution trajectories do not converge to a boundary equilibrium "too fast" (if it should converge at all), see Theorem \ref{theo:boundary}. To explain this approach further, we remark at first that our proof of deriving the entropy-entropy dissipation inequality \eqref{MainEstimate} from the finite dimensional inequality \eqref{last_1_re} is independent of the presence of boundary equilibria.} {Thus, instead of trying (or rather failing) to prove \eqref{last_1_re} as a pure functional inequality, we look for a generalisation with a \emph{time-dependent} coefficient $H_1(t)$ along the trajectories of solutions, {where $H_1(t)$ may tend to zero in case a solution trajectory would converge to a boundary equilibrium.}
Therefore, we look for a modified entropy-entropy dissipation inequality along solutions $\cc(x,t)$ of \eqref{e0}--\eqref{Nin} of the following form (which is no longer a pure functional inequality like \eqref{MainEstimate})}
\begin{equation}\label{modified_entropy}
	\mathcal{D}(\cc(t)) \geq \lambda(t)\, \mathcal{E}(\cc(t)|\ww)
\end{equation}
with $\lambda(t) = \frac 12\min\{\lambda_1, K_2H_1(t)/K_1\}$ where $K_1$ and $K_2$ are given in \eqref{K1} and \eqref{K2}. Intuitively, the time dependent function $\lambda(t)$ (which may decay to $0$ as $t\to\infty$) gives a lower bound for the entropy dissipation $\mathcal{D}(\cc(t))$ or equivalently for the convergence of a trajectory towards a boundary equilibrium (where $\mathcal{E}(\cc(t)|\ww)$ remains bounded below). Therefore, if $\lambda(t)$ satisfies $\int_0^{+\infty}\lambda(s)ds = +\infty$ or equivalently the function $H_1(t)$ satisfies $\int_0^{+\infty}H_1(s)ds = +\infty$, then it follows from Gronwall's inequality and the weak entropy-entropy dissipation law \eqref{weak} that
\begin{equation}\label{modified_convergence}
	\mathcal{E}(\cc(t)|\ww) \leq \mathcal{E}(\cc_0|\ww)e^{-\int_0^t \lambda(s)ds} \rightarrow 0 \quad \text{ as } t\to\infty.
\end{equation}
{By using this (so far non-exponential) convergence to the complex balanced equilibrium $\ww$, we obtain the $L^1$-instability of the boundary equilibria. In return, this instability allows to show an entropy-entropy dissipation estimate of the form \eqref{MainEstimate} 
on a reduced domain of states, which is strictly bounded away from the boundary equilibria.  Thus, we recover exponential convergence to the complex balanced equilibrium $\ww$ after a sufficiently large time. Our second main result reads as follows:}
\begin{theorem}[Conditional convergence to equilibrium for complex balanced reaction-diffusion systems with boundary equilibria]\label{theo:boundary}\hfill\\
Let $\Omega$ be a bounded domain in $\mathbb R^n$ with Lipschitz boundary $\partial\Omega$. Assume that the diffusion matrix $\mathbb D$ is positive definite, i.e. $d_i>0$ for all $i=1,\ldots, N$. Moreover, we assume that system \eqref{e0}--\eqref{Nin} is complex balanced. Consequently, for each positive initial mass vector $M \in \mathbb R^m_{>0}$ there exists a unique positive complex balanced equilibrium $\ww\in \mathbb R^N_{>0}$. Note that the system may possess  (possibly infinitely) many boundary equilibria.
		
		\medskip
		Let $\cc(x,t)$ be a renormalised solution to \eqref{e0}--\eqref{Nin} with initial data satisfying $\mathbb Q\,\overline{\cc}_0 = M$ and $\mathcal{E}(\cc_0|\ww) <+\infty$. 
{Note that any such renormalised solution  
satisfies the mass conservation laws $\mathbb Q\,\overline{\cc}(t) = M$, \cite{FT15,Fis16}.}		
Assume that there exists a function $H_1: [0,+\infty) \to [0,+\infty)$ with the property $\int_0^{+\infty}H_1(s)ds = +\infty$, such that 
		\begin{equation}\label{eq:finite_time}
			\sum_{r=1}^{|\mathcal R|}\left[\sqrt{\frac{\overline{\cc}(t)}{\ww}}^{\yy_r} - \sqrt{\frac{\overline{\cc}(t)}{\ww}}^{\yy_r'}\right]^2 \geq H_1(t)\sum_{i=1}^{N}\left(\sqrt{\frac{\overline{c_i}(t)}{c_{i,\infty}}} - 1\right)^2, \qquad \text{for a.a. }t\ge 0.
		\end{equation}
		Then, the renormalised solution $\cc(x,t)$ converges exponentially to the positive complex balanced equilibrium $\ww$ in the $L^1$-norm with  a rate, which can be explicitly computed in terms of the function $H_1$, the domain $\Omega$, the diffusion matrix $\mathbb D$, the stoichiometric coefficients $\yy\in \mathcal C$, the initial mass $M$, the complex balanced equilibrium $\ww$ and the reaction rate constants $k_r$.
\end{theorem}

%\begin{remark}
The main progress of Theorem \ref{theo:boundary} is that the question of convergence 
to equilibrium for complex balanced reaction-diffusion systems {\it with boundary equilibria} is reduced 
to proving the finite dimensional inequality \eqref{eq:finite_time}. Moreover, if the function $H_1(t)$ is explicitly computable (i.e. for some specific systems), then the rate of equilibration of the renormalised solution $\cc(x,t)$ to \eqref{e0}--\eqref{Nin} can also be computed explicitly. However, proving \eqref{eq:finite_time} for general systems with boundary equilibria remains a difficult problem since it requires suitable estimates on renormalised solutions, more precisely, on the behaviour of the $L^1$-norm of renormalised solutions near the boundary $\partial\mathbb R_{>0}^N$, which is already a hard problem for ODE systems with boundary equilibria. Nevertheless, we will show in Subsection \ref{sec:system_boundary} how to apply  Theorem \ref{theo:boundary} to specific systems.
%\end{remark}

\begin{remark}[Towards a Global Attractor Conjecture for Reaction-Diffusion Systems]\hfill\\
It is worthwhile to remark on the key assumption $\int_0^{+\infty}H_1(s)ds = +\infty$. Note first that if $H_1(t)$ is defined as the fraction of the left-hand-side sum and the right-hand-side sum of 
\eqref{eq:finite_time}, then $H_1(t)=0$ if and only if 
$\overline{\cc}(t)$ is a boundary equilibrium and otherwise bounded (if 
$\overline{\cc}(t)=\ww$, then this $H_1(t)$ extends continuously to a positive constant). Hence, it is equivalent to consider in Theorem \ref{theo:boundary} the assumption $\int_{t_1}^{+\infty}H_1(s)ds = +\infty$ for a time $t_1$ arbitrarily large.

Secondly, note that since \eqref{eq:finite_time} constitutes a finite dimensional inequality for the spatial averages $\overline{\cc}(t)$, one could conjecture to prove 
\eqref{eq:finite_time} by assuming the Global Attractor Conjecture for the corresponding ODE system \eqref{e00},
%, i.e.  
%\begin{equation}\label{oo}
%\frac{d}{d t}\boldsymbol{u} = \mathbf{R}(\boldsymbol{u})
%\end{equation}
%with $\mathbf{R}(\boldsymbol{u})$ is defined in \eqref{Reaction} 
see \cite{Cra15} for a proof under review of the GAC for complex balanced ODE systems.

Indeed at a time $t_1>0$, consider the spatial averages $\overline{\cc}(t_1)$ as initial data $\boldsymbol{u}(t_1)= \overline{\cc}(t_1)$ of \eqref{e00}.
Then, the ODE Global Attractor Conjecture for \eqref{e00} should imply (via a contradiction argument) the existence of
$H_{1}^{\mathrm{ODE}}(t)$ with $\int_{t_1}^{+\infty}H^{\mathrm{ODE}}_1(s)ds = +\infty$  such that (following \eqref{eq:finite_time})
\begin{equation*}
\sum_{r=1}^{|\mathcal R|}\left[\sqrt{\frac{\uu(t)}{\ww}}^{\yy_r} - \sqrt{\frac{\uu(t)}{\ww}}^{\yy_r'}\right]^2 \geq H_1^{ODE}(t)\sum_{i=1}^{N}\left(\sqrt{\frac{\uu(t)}{c_{i,\infty}}} - 1\right)^2, \qquad \text{for a.a. }t\ge t_1,
\end{equation*}
for $\uu$ to be the solution of \eqref{e00}, since the ODE system \eqref{e00} shares the same complex balanced equilibria as the PDE system. 
Moreover, formal estimates seem to suggest that it is  possible to 
establish bounds on $H_1(t)$ via $H_{1}^{\mathrm{ODE}}(t)$ 
on a sufficiently small time interval $(t_1,t_2)$ provided 
that there is a good comparison between of the evolution of the 
ODE system $\uu(t)$ and the evolution of the PDE system
$\cc(t)$ via its spatial averages $\overline{\cc}(t)$.
Next at time $t_2$, one restarts the ODE evolution \eqref{e00} with a second set of initial data $\boldsymbol{u}(t_2)= \overline{\cc}(t_2)$ and use that also this ODE system satisfies the GAC and yields another function $H_{2}^{\mathrm{ODE}}(t)$
on a time interval $(t_2,t_3)$ and so forth. Assuming that the 
evolution of $\overline{\cc}(t)$ converges sufficiently fast to 
these family of related ODE solutions, is seems possible to prove 
a statement like ODE GAC  implies 
GAC for the PDE systems. 
%{\color{red}(modulo some special situations 
%where the ODE systems is at a boundary equilibrium, but the PDE systems not (yet), e.g. that the PDE solution still diffuses the non-zero species but the averages $\overline{\cc}(t)$ are already at the boundary equilibrium).}

However, the problem of deriving good convergence estimates on the difference between the ODE system \eqref{e00} and the evolution of the spatial averages $\overline{\cc}(t)$ seems to be (at least) as hard as understanding directly the evolution of $\cc(t)$. First, 
the non-convexity of $\mathbf{R}(\uu)$ prevents any direct 
comparison between $\frac{d}{d t}\uu=\mathbf{R}(\uu) 
\neq \overline{\mathbf{R}(\cc)} =  \frac{d}{d t} \overline{\cc}(t)$.
Moreover, the evolution of the difference $\cc(t)-\uu(t)$ is not non-negative and doesn't seem to feature an entropy functional. 
Hence, it seems that in order to derive estimates on the difference $\overline\cc(t)-\uu(t)$, one is brought back to understanding the equilibration of $\cc(t)$, which is the problem to solve at first. 
\end{remark}

%\textcolor{red}{Comparison ODE system, PDE system? Elementary estimates to compare (16) between ODE and PDE setting?}

\medskip
\noindent{\bf Notations:} Throughout this paper, we will use the following set of convenient notations:
\medskip
\begin{itemize}
	\item[$\bullet$] Capital letters for square roots of corresponding normal letter, that is $C_i = \sqrt{c_i}$ or $C_{i,\infty} = \sqrt{c_{i,\infty}}$.
	\item[$\bullet$] The usual norm in $L^2(\Omega)$ is denoted by $\|\cdot\|$, i.e.
		\begin{equation*}
			\|f\| = \left(\int_{\Omega}|f(x)|^2dx\right)^{1/2}.
		\end{equation*}
	\item[$\bullet$] For two vectors $\boldsymbol{y} = (y_1, \ldots, y_N)$ and $\boldsymbol{z} = (z_1, \ldots, z_N)$ in $\mathbb R^{N}$ with $z_i\not=0$ for all $i=1,\ldots, N$, we write
		\begin{equation*}
			\frac{\boldsymbol{y}}{\boldsymbol{z}} = \left(\frac{y_1}{z_1}, \ldots, \frac{y_N}{z_N}\right).
		\end{equation*}
	\item[$\bullet$] For a function $f: \mathbb R\to \mathbb R$ and a vector $\boldsymbol{y}\in \mathbb R^N$, we denote by
		\begin{equation*}
			f(\boldsymbol{y}) = (f(y_1), \ldots, f(y_N))\in \mathbb R^N.
		\end{equation*}
		For example,
		\begin{equation*}
			\sqrt{\frac{\boldsymbol{y}}{\boldsymbol{z}}} = \left(\sqrt{\frac{y_1}{z_1}}, \ldots, \sqrt{\frac{y_N}{z_N}}\,\right).
		\end{equation*}
\end{itemize}
\medskip
{\bf Organisation of the paper:} In section \ref{sec:2}, we present the proof of Theorem \ref{theo:main} and show its application to a reversible enzyme reaction. The proof of Theorem \ref{theo:boundary} and applications to networks with boundary equilibria will be presented in Section \ref{sec:3}.

\section{Proof of Theorem \ref{theo:main} and Applications}\label{sec:2}
\subsection{Proof of Theorem \ref{theo:main}}\hfill
\vskip .5cm
\noindent{\bf Renormalised solutions.}\\
The existence of global solutions for general reaction-diffusion systems of the form \eqref{e0}--\eqref{Reaction} is a challenging question as already mentioned in the introduction. A huge amount of references in the literature dealt in one or the other way with this issue and provided partial existence results of weak or strong solutions under suitable assumptions, for instance the size or structure of the system, the smallness of the space-dimension or the closeness of diffusion coefficients. Recently, global existence of renormalised solution (inspired by the concept of renormalised solutions for Boltzmann equation)  it was proved by Fischer in \cite{Fis15} for systems of the form \eqref{e0}--\eqref{Nin} and even more general dissipative systems. Moreover, for the mass-action law systems \eqref{e0}--\eqref{Nin}, it was very recently proven in \cite{Fis16} that all renormalised solutions according to the following definition satisfy the weak entropy-entropy dissipation law \eqref{weak} and the conservation laws \eqref{12ter}.
	
	\begin{theorem}[Renormalised solutions to \eqref{e0}--\eqref{Nin}, \cite{Fis15,Fis16}]\label{renormalised}\hspace*{\fill}\\
		Let $\Omega$ be a bounded domain in $\mathbb R^n$ with Lipschitz boundary $\partial\Omega$. Assume that the diffusion matrix $\mathbb D$ is positive definite, i.e. $d_i>0$ for all $i=1,\ldots, N$. Moreover, we assume that system \eqref{e0} is complex balanced and thus possesses the  entropy dissipation structure \eqref{FreeEnergy_PDE} and \eqref{ED_PDE}. 
	
		\medskip
		Then, for any nonnegative initial data $\cc_0: \Omega\to \mathbb R^N$ having finite relative entropy $\mathcal{E}(\cc_0|\ww)<+\infty$, there exists a global renormalised solution $\cc(x,t) = (c_1(x,t), \ldots, c_N(x,t))$ to \eqref{e0}--\eqref{Nin}: That is $c_i\log c_i \in L^{\infty}_{loc}(\mathbb R_+; L^1(\Omega))$ and $\sqrt{c_i} \in L^2_{loc}(\mathbb R_+; H^1(\Omega))$ and for any smooth function $\xi: \mathbb R_+^N \rightarrow \mathbb R$ with compactly supported derivative $\nabla \xi$ and every $\psi \in C^{\infty}(\overline{\Omega}\times \mathbb R_+)$ there holds
		\begin{equation*}
		\begin{aligned}
		\int_{\Omega}\xi(\cc(\cdot, T))\psi(\cdot, T)&\,dx - \int_{\Omega}\xi(\cc_0)\psi(\cdot, 0)\,dx - \int_{0}^{T}\!\!\int_{\Omega}\xi(\cc)\frac{d}{dt}\psi\, dx dt\\
		= &-\sum_{i,j=1}^{N}\int_{0}^{T}\!\!\int_{\Omega}\psi \,\partial_{i}\partial_{j}\xi(\cc)(d_i\nabla c_i)\!\cdot\! \nabla c_j\, dx dt\\
		& - \sum_{i=1}^{N}\int_{0}^{T}\!\!\int_{\Omega}\partial_{i}\xi(\cc)(d_i\nabla c_i)\!\cdot\! \nabla \psi\, dx dt
		+ \sum_{i=1}^{N}\int_{0}^{T}\!\!\int_{\Omega}\partial_i\xi(\cc)R_i(\cc)\psi\, dx dt
		\end{aligned}
		\end{equation*}
		for almost every $T>0$. 
		
		\medskip
		Moreover, any renormalised solution $\cc(x,t)$ to \eqref{e0}--\eqref{Nin} satisfies the weak entropy-entropy dissipation law
		\begin{equation*}
			\mathcal E(\cc(t)|\ww) + \int_{s}^{t}\mathcal D(\cc(\tau))d\tau \leq \mathcal{E}(\cc(s)|\ww)
		\end{equation*}
		for almost all $t\geq s \geq 0$, and the mass conservation laws, i.e.
		\begin{equation*}
			\mathbb Q\,\overline{\cc}(t) = \mathbb Q\,\overline{\cc_0} \quad \text{ for a.a. } \quad t>0.
		\end{equation*}
	\end{theorem}
	\vskip .5cm
\noindent{\bf Preliminary estimates.}\\
We present in this part some useful preliminary estimates which are needed for the sequel proofs.

The following Csisz\'{a}r-Kullback-Pinsker type inequality shows that convergence to equilibrium in relative entropy implies convergence to equilibrium in $L^1$-norm. For its proof, even in more general settings, we refer the reader to e.g. \cite{Arnold,DeFe06,DeFe08,DFT16}.
\begin{lemma}[A Csisz\'{a}r-Kullback-Pinsker type inequality]\label{lem:CKP}
Fix an initial mass vector $M\in \mathbb R_{>0}^m$. Then, there exists a constant $C_{\mathrm{CKP}}$ depending only on $\Omega$, $M$ and $\ww$ such that for all measurable $\cc: \Omega \to \mathbb R_+^{N}$ satisfying the mass conservation $\mathbb Q\,\overline{\cc} = M$, there holds
	\begin{equation*}
	\mathcal{E}(\cc|\ww) \geq C_{\mathrm{CKP}}\sum_{i=1}^{N}\|c_i - c_{i,\infty}\|_{L^1(\Omega)}^2.
	\end{equation*}
\end{lemma}

\begin{lemma}[Additivity of relative entropy]\label{lem:1}
	For all measurable $\cc:\Omega \rightarrow \mathbb R_+^N$ with finite relative entropy $\mathcal{E}(\cc|\ww)<+\infty$  holds
	\begin{equation*}
		\mathcal{E}(\cc|\ww) = \mathcal{E}(\cc|\overline{\cc}) + \mathcal{E}(\overline{\cc}|\ww),
	\end{equation*}
where we recall $\overline{\cc} = (\overline{c_1}, \ldots, \overline{c_N})$ with $\overline{c_i} = \int_{\Omega}c_idx$.
\end{lemma}
\begin{proof}
	The proof follows from direct computations, hence we omit it here.
\end{proof}

Lemma \ref{lem:1} allows to prove the entropy-entropy dissipation inequality \eqref{MainEstimate} by estimating $\mathcal{E}(\cc|\overline{\cc})$ and $\mathcal E(\overline{\cc}|\ww)$ separately. The first part $\mathcal{E}(\cc|\overline{\cc})$ can be easily controlled by $\mathcal{D}(\cc)$ thanks to the Logarithmic Sobolev inequality as in the following
\begin{lemma}\label{lem:2}
For all measurable $\cc:\Omega \rightarrow \mathbb R_+^N$ with finite relative entropy $\mathcal{E}(\cc|\ww)<+\infty$  holds
	\begin{equation*}
		\mathcal{D}(\cc) \geq \lambda_1\mathcal{E}(\cc|\overline{\cc})
	\end{equation*}
	for $\lambda_1=\min_{i=1,\ldots, N}\{d_i\}\,C_{\mathrm{LSI}}$ where $C_{\mathrm{LSI}}$ is the best constant in the Logarithmic Sobolev inequality.
\end{lemma}
\begin{proof}
	By using the Logarithmic Sobolev inequality
	\begin{equation*}
		\int_{\Omega}\frac{|\nabla f|^2}{f}dx \geq C_{\mathrm{LSI}}\int_{\Omega}f\log{\frac{f}{\overline{f}}}dx
	\end{equation*}
	for all  nonnegative $f\in H^1(\Omega)$, we estimate
	\begin{equation*}
		\mathcal{D}(\cc) \geq \min_{i=1,\ldots,N}\{d_i\}\,C_{\mathrm{LSI}}\sum_{i=1}^{N}\int_{\Omega}c_i\log{\frac{c_i}{\overline{c_i}}}dx = \min_{i=1,\ldots,N}\{d_i\}\,C_{\mathrm{LSI}}\,\mathcal{E}(\cc|\overline{\cc}).
	\end{equation*}
\end{proof}

{Thanks to the Lemmas \ref{lem:1} 
and \ref{lem:2}, 
the remaining part of this section is dedicated to control the second part $\mathcal{E}(\overline{\cc}|\ww)$ of the relative entropy $\mathcal{E}(\cc|\ww)$. Note that such a control has 
to quantify the system behaviour of the reacting concentrations 
$\cc$ as well as the conservation laws $\mathbb Q\,\overline{\cc}(t) = M =\mathbb Q\,\overline{\cc_0}$. Therefore, 
the control of $\mathcal{E}(\overline{\cc}|\ww)$ is much more challenging and technical.}

We first show that $\mathcal{E}(\overline{\cc}|\ww)$ is bounded above by the right hand side of the finite dimensional inequality \eqref{last_1_re}.
\begin{lemma} %[$\mathcal{E}(\overline{\cc}|\ww)$ is bounded by the right hand side of \eqref{last_1_re}]
\label{lem:3}
	For any measurable $\cc: \Omega \to \mathbb R^N$ satisfying $\mathcal{E}(\overline{\cc}|\ww)\leq K$, it holds that 
	\begin{equation*}
		\mathcal{E}(\overline{\cc}|\ww) \leq K_1\sum_{i=1}^{N}\left(\sqrt{\frac{\overline{c_i}}{c_{i,\infty}}} - 1\right)^2
	\end{equation*}
	for an {\normalfont explicit} constant $K_1>0$ depending on $K$ and $\ww$ (see \eqref{K1}).
\end{lemma}
\begin{proof}
	First, by using the elementary inequalities
	\[\log(x/y) - x + y \geq \left(\sqrt{x} - \sqrt{y}\right)^2 \geq \frac 12 x - y\]		
	we easily deduce from $\mathcal{E}(\overline{\cc}|\ww)\leq K$ that 
	\begin{equation}\label{Ktilde}
		\overline{c_i} \leq \widetilde{K}:= 2\left(K + \sum_{i=1}^{N}c_{i,\infty}\right), \qquad \text{for all } i=1,2,\ldots, N. 
	\end{equation}
	Next, we introduce the function 
	\begin{equation*}
		\Phi(z) = \frac{z\log z - z + 1}{(\sqrt{z} - 1)^2}
	\end{equation*}
	which is continuous on $[0,\infty)$ with the extensions $\Phi(0) = \lim_{z\to 0}\Phi(z) = 1$ and $\Phi(1) = \lim_{z\to 1}\Phi(z) = 2$, and monotone increasing. By using now the bound $\overline{c_i} \leq \widetilde K$,  we can estimate
	\begin{equation*}
		\mathcal{E}(\overline{\cc}|\ww) = \sum_{i=1}^{N}\left(\overline{c_i}\log{\frac{\overline{c_i}}{c_{i,\infty}}} - \overline{c_i} + c_{i,\infty}\right) =  \sum_{i=1}^{N}c_{i,\infty}\Phi\left(\frac{\overline{c_i}}{c_{i,\infty}}\right) \left(\sqrt{\frac{\overline{c_i}}{c_{i,\infty}}} - 1\right)^2\leq K_1\sum_{i=1}^{N}\left(\sqrt{\frac{\overline{c_i}}{c_{i,\infty}}} - 1\right)^2
	\end{equation*}
	with
	\begin{equation}\label{K1}
		K_1 = \max_{i=1,\ldots,N}c_{i,\infty}\Phi\left(\frac{\widetilde K}{c_{i,\infty}}\right).
	\end{equation}
\end{proof}
\vskip .5cm
%\noindent{\bf Proof of main results.}\\
By using Lemmas \ref{lem:1}, \ref{lem:2} and \ref{lem:3}, 
where the latter establishes the right hand side of \eqref{last_1_re}, our proof of Theorem  \ref{theo:main} still requires i) to control the left hand side of the finite dimensional inequality \eqref{last_1_re} in terms of $\mathcal{D}(\cc)$, and ii) to prove \eqref{last_1_re}. These will be done in Lemmas \ref{lem:5}, \ref{lem:4} and Lemma \ref{lem:finite} respectively.

	As the first step, we observe that the entropy dissipation $\mathcal{D}(\cc)$ is a combination of the diffusion and reaction processes of the system \eqref{e0}, see \eqref{ED_PDE}. The reaction term seems hard to control due to the non-convex nonlinearities (with arbitrary high order polynomials) and the very low regularity of  renormalised solutions. In fact, we will not show that the entropy dissipation $\mathcal D(\cc)$ is bounded for 
renormalised solutions, but only that it constitutes an upper bound even while potentially unbounded. We will prove in the following lemma that, with the help of the diffusion terms, the reaction part is bounded below by "reactions of averaged concentrations". Herein, we recall the convention of square roots $C_i = \sqrt{c_i}$ and $C_{i,\infty} = \sqrt{c_{i,\infty}}$ and denote by $\CC = (C_1, \ldots, C_N)$ and $\overline{\CC} = (\overline{C_1},\ldots, \overline{C_N})$ with $\overline{C_i} = \int_{\Omega}C_idx$.
	\begin{lemma}\label{lem:5} 
For any measurable $\cc: \Omega \to \mathbb R^N_+$ holds that
	\begin{equation}\label{e3}
		\mathcal{D}(\cc) \geq K_3\left(\sum_{i=1}^{N}\|\nabla C_i\|^2 + \sum_{r=1}^{|\mathcal R|}\left[\frac{\overline{\CC}^{\yy_r}}{\WW^{\yy_r}} - \frac{\overline{\CC}^{\yy_r'}}{\WW^{\yy_r'}}\right]^{\!2}\right)
	\end{equation}
	for \emph{explicit} constant $K_3>0$ (see \eqref{K3}). %, where $\CC = (C_1, \ldots, C_N)$ and $\WW = (C_{1,\infty}, \ldots, C_{N,\infty})$, recalling the convention for square roots $C_i = \sqrt{c_i}$ and $C_{i,\infty} = \sqrt{c_{i,\infty}}$.
	\end{lemma}
	\begin{proof}
	By using the identity $\nabla \sqrt{c_i} = \nabla c_i/(2\sqrt{c_i})$ and the elementary inequality $\Psi(x,y) =  x\log(x/y) - x + y \geq (\sqrt{x} - \sqrt{y})^2$, we estimate \eqref{ED_PDE} as
	\begin{equation}\label{e2}
		\mathcal{D}(\cc) \geq \sum_{i=1}^{N}4d_i\|\nabla C_i\|^2 + \sum_{r=1}^{|\mathcal R|}k_r\ww^{\yy_r}\biggl\|\frac{\CC^{\yy_r}}{\WW^{\yy_r}} - \frac{\CC^{\yy_r'}}{\WW^{\yy_r'}}\biggr\|^2.
	\end{equation}

	To prove \eqref{e3}, we use similar arguments to  \cite{FL15,DFT16,FT15}. %For the sake of completeness, we will recall a sketch of the proof here. 
%	Note also that the same procedure will be used later to prove an auxiliary lemma \ref{lem:Depsilon} for approximating solutions sequence. 
	Fix a constant $L>0$. The proof uses a domain decomposition corresponding to the deviation of $C_i$ around the averages $\overline{C_i}$, i.e. by denoting $\delta_i(x) = C_i(x) - \overline{C}_i$, we consider the decomposition 
				\begin{equation*}
					\Omega = S \cup S^{c},
				\end{equation*}
	where $S = \{x\in\Omega: |\delta_i(x)| \leq L \text{ for all } i = 1,\ldots, N \}$ and $S^c = \Omega \backslash S$. {We will see that on $S$ the reaction part is crucial while the diffusion part is sufficient on $S^{c}$}. 
	
	On the set $S$, by using the bounds $|\delta_i(x)| \leq L$ and $\overline{C_i} = \overline{\sqrt{c_i}} \leq \sqrt{\overline{c_i}}\leq \sqrt{\widetilde{K}}$ (by Jensen's inequality and \eqref{Ktilde}), as well as Taylor expansion of terms like 
	\[
		(\overline{C}_i + \delta_i)^{y_{r,i}} = \overline{C_i}^{y_{r,i}} + \widetilde{R}_i\delta_i \quad \text{ with } \quad \widetilde{R}_i = y_{r,i}(\theta \overline{C_i} + (1-\theta)\delta_i)^{y_{r,i}-1} \text{ for some } \theta \in (0,1)
	\]
	and the elementary inequalities $(x-y)^2 \geq x^2/2 - y^2$ and $(\sum_{i=1}^{N}z_i)^2 \leq N\sum_{i=1}^{N}z_i^2$, we can first show that 
	\begin{equation}\label{ee1}
	\begin{aligned}
	\sum_{r=1}^{|\mathcal R|}k_r\ww^{\yy_r}
	\left\|\frac{\mathbf{C}^{\yy_r}}{\mathbf{C}_{\infty}^{\yy_r}} - \frac{\mathbf{C}^{\yy_r'}}{\mathbf{C}_{\infty}^{\yy_r'}}\right\|_{L^2(S)}^2 &\geq  \underbrace{\min_{1\le r\le|\mathcal R|}\{k_r\ww^{\yy_r}\}}_{=: 2\beta_1}\sum_{r=1}^{|\mathcal R|}
	\left\|\frac{\mathbf{C}^{\yy_r}}{\mathbf{C}_{\infty}^{\yy_r}} - \frac{\mathbf{C}^{\yy_r'}}{\mathbf{C}_{\infty}^{\yy_r'}}\right\|_{L^2(S)}^2\\
	&= 2\beta_1\sum_{r=1}^{|\mathcal R|}\left\|\prod_{i=1}^{N}\left(\frac{\overline{C_i}+\delta_i}{C_{i,\infty}}\right)^{y_{r,i}} - \prod_{i=1}^{N}\left(\frac{\overline{C_i}+\delta_i}{C_{i,\infty}}\right)^{y_{r,i}'}\right\|_{L^2(S)}^2\\
	&\geq  \beta_1
	\sum_{r=1}^{|\mathcal R|}
	\left[\frac{\overline{\mathbf{C}}^{\yy_r}}{\mathbf{C}_{\infty}^{\yy_r}} - \frac{\overline{\mathbf{C}}^{\yy_r'}}{\mathbf{C}_{\infty}^{\yy_r'}}\right]^2|S| - \beta_2\sum_{i=1}^{N}\|\delta_i\|_{L^2(S)}^2
	\end{aligned}
	\end{equation}
%	\begin{equation}\label{ee1}
%		\sum_{r=1}^{|\mathcal R|}k_r\ww^{\yy_r}
%		\left\|\frac{\mathbf{C}^{\yy_r}}{\mathbf{C}_{\infty}^{\yy_r}} - \frac{\mathbf{C}^{\yy_r'}}{\mathbf{C}_{\infty}^{\yy_r'}}\right\|_{L^2(S)}^2 \geq  \beta_1
%		\sum_{r=1}^{|\mathcal R|}
%		\left[\frac{\overline{\mathbf{C}}^{\yy_r}}{\mathbf{C}_{\infty}^{\yy_r}} - \frac{\overline{\mathbf{C}}^{\yy_r'}}{\mathbf{C}_{\infty}^{\yy_r'}}\right]^2|S| - \beta_2\sum_{i=1}^{N}\|\delta_i\|_{L^2(S)}^2
%	\end{equation}
	with
	\begin{equation}\label{beta12}
		\begin{gathered}
		\beta_1 = \frac 12 \min_{1\le r\le|\mathcal R|}\{k_r\ww^{\yy_r}\},\\
		{\beta_2 = 2\beta_1N|\mathcal R| \max_{\yy\in \mathcal C, 1\leq i,j\leq N}\left\{\frac{1}{c_{i,\infty}^{y_i-1}\prod\limits_{\ell=1}^{N-i}c_{i,\infty}^{y_{\ell}}}\left(\sqrt{\widetilde{K}}^{y_{i}} + Ly_j(\sqrt{\widetilde{K}}+L)^{y_j-1}\right)^{N-i}y_i(\sqrt{\widetilde{K}} + L)^{y_i - 1}\right\}.}
		\end{gathered}
	\end{equation}

	On the other hand, on $S^{c}$, by using the lower bound $|\delta_i|\ge L$ for some $1\leq i\leq N$, the upper bound {$\overline{C}_i\le \sqrt{\widetilde{K}}$}, and Poincar\'e's inequality,  it follows that
	\begin{equation}\label{ee2}
	\sum_{i=1}^{N}d_i\|\nabla C_i\|^2 \geq \underbrace{C_P\min_{i}\{ d_i\}}_{=:\beta_3}\sum_{i=1}^{N}\|\delta_i\|^2 \geq \beta_3 L |S^c| \geq  \beta_4\sum_{r=1}^{|\mathcal R|}\left[\frac{\overline{\mathbf{C}}^{\yy_r}}{\mathbf{C}_{\infty}^{\yy_r}} - \frac{\overline{\mathbf{C}}^{\yy_r'}}{\mathbf{C}_{\infty}^{\yy_r'}}\right]^2|S^{c}|,
	\end{equation}
	with
	\begin{equation}\label{beta34}
		\begin{gathered}
			\beta_3 = C_P\min_{i}\{ d_i\},\qquad 
			\beta_4 = N\beta_3L^2\left[\max_{1\le r\le |\mathcal R|}2\left(\frac{{\widetilde{K}}^{|\yy_r|}}{(\WW^{\yy_r})^2} + \frac{{\widetilde{K}}^{|\yy_r'|}}{(\WW^{\yy_r'})^2}\right) \right]^{-1}
		\end{gathered}
	\end{equation}
	where $|\yy| = \sum_{i=1}^{N}y_i$ for any $\yy\in \mathcal C$. Note that all the constants $\beta_1, \beta_2, \beta_3, \beta_4$ defined in \eqref{beta12} and \eqref{beta34} are independent of $S$ and $S^c$. By combining \eqref{ee1} and \eqref{ee2}, we can estimate with any $\gamma \in (0,1)$ and the Poincar\'{e} inequality
	\begin{equation*}\label{ee3}
		\begin{aligned}
			\mathcal{D}(\cc) &\geq 2\min_{j}\{d_j\}\sum_{i=1}^{N}\|\nabla C_i\|^2\\
			&\qquad  +\left(\beta_3\sum_{i=1}^{N}\|\delta_i\|^2 + \sum_{i=1}^{N}d_i\|\nabla C_i\|^2 + \gamma \sum_{r=1}^{|\mathcal R|}k_r\ww^{\yy_r}\biggl\|\frac{\CC^{\yy_r}}{\WW^{\yy_r}} - \frac{\CC^{\yy_r'}}{\WW^{\yy_r'}}\biggr\|_{L^2(S)}^2\right)\\
			&\geq 2\min_{j}\{d_j\}\sum_{i=1}^{N}\|\nabla C_i\|^2 + (\beta_3 - \gamma\beta_2)\sum_{i=1}^{N}\|\delta_i\|^2 + \min\{\gamma\beta_1,\beta_4\}\sum_{r=1}^{|\mathcal R|} \left[\frac{\overline{\mathbf{C}}^{\yy_r}}{\mathbf{C}_{\infty}^{\yy_r}} - \frac{\overline{\mathbf{C}}^{\yy_r'}}{\mathbf{C}_{\infty}^{\yy_r'}}\right]^2.
		\end{aligned}
	\end{equation*}
since $|S|+|S^{c}|=|\Omega|=1$. By choosing $\gamma = \frac 12\min\{1; \beta_2^{-1}\beta_3\}$, we obtain \eqref{e3} with
	\begin{equation}\label{K3}
		K_3 = \min\left\{2\min_{j}\{d_j\}, \min\left\{\frac 12 \min\{1, \beta_2^{-1}\beta_3\}\beta_1, \beta_4\right\}\right\}
	\end{equation}
	where $\beta_1, \beta_2, \beta_3$ and $\beta_4$ are defined in \eqref{beta12} and \eqref{beta34}.
	\end{proof}
	
	%{\color{blue}
	\begin{remark}
The constant $L$ in Lemma \ref{lem:5} can be chosen arbitrary. One certainly can choose $L$ in order to optimise (i.e. maximise) the constant $K_3$ in \eqref{K3}. This may help to improve the rate of convergence. However, due to the multistage proof of the entropy-entropy dissipation inequality, the estimated rates are not optimal.
\end{remark}
	%}
	Now we are able to control the left hand side of \eqref{last_1_re} by $\mathcal {D}(\cc)$.
\begin{lemma}\label{lem:4}
For any measurable $\cc: \Omega\to \mathbb R_+^N$ satisfying $\mathcal{E}(\overline{\cc}|\ww) \leq K$ and $\mathbb Q\,\overline{\cc} = M$, there exists an {\normalfont explicit} constant $K_2>0$ (see \eqref{K2}) such that 
	\begin{equation*}
		\mathcal{D}(\cc) \geq K_2\sum_{r=1}^{|\mathcal R|}\left[\sqrt{\frac{\overline{\cc}}{\ww}}^{\yy_r} - \sqrt{\frac{\overline{\cc}}{\ww}}^{\yy_r'}\right]^2.
	\end{equation*}
\end{lemma}
\begin{remark}
Lemma \ref{lem:4} is a crucial step in proving Theorem \ref{theo:main}. 
{As mentioned in the introduction, we are here able to remove a technical assumption on cases when the $L^1$-norm of the concentrations approaches the boundary $\partial\mathbb R_+^N$, which was needed in \cite[Theorem 1.4]{FT15}. The key observation is the remainder estimate  \eqref{estimateRemainder}.} Note that this idea was also used in \cite[Lemma A.5]{HHMM} for energy-reaction-diffusion systems.
\end{remark}
\begin{proof}
	By denoting
	\begin{equation*}
		\delta_i(x) = C_i(x) - \overline{C_i}, \qquad \text{ for } x \in \Omega
	\end{equation*}
	for $i=1,\ldots, N$, we have $\|\delta_i\|^2 = \overline{C_i^2} - \overline{C_i}^2$, which leads to
	\begin{equation*}
		\overline{C_i} = \sqrt{\overline{C_i^2}} - \frac{\|\delta_i\|^2}{\sqrt{\overline{C_i^2}} + \overline{C_i}}
	\end{equation*}
	and consequently
	\begin{equation}\label{ansatz}
		\frac{\overline{C_i}}{C_{i,\infty}} = \frac{\sqrt{\overline{C_i^2}}}{C_{i,\infty}} - \frac{\|\delta_i\|^2}{C_{i,\infty}\Bigl(\sqrt{\overline{C_i^2}} + \overline{C_i}\Bigr)} = \sqrt{\frac{{\overline{c_i}}}{c_{i,\infty}}} - Q(C_i)\|\delta_i\|
	\end{equation}
	with
	\begin{equation*}
		Q(C_i) = \frac{\|\delta_i\|}{C_{i,\infty}\Bigl(\sqrt{\overline{C_i^2}} + \overline{C_i}\Bigr)}.
	\end{equation*}
	Note that $Q(C_i)\geq 0$ and that
	\begin{equation}\label{estimateRemainder}
		Q(C_i)^2 = \frac{\|\delta_i\|^2}{C_{i,\infty}^2\Bigl(\sqrt{\overline{C_i^2}} + \overline{C_i}\Bigr)^2} = \frac{\sqrt{\overline{C_i^2}} - \overline{C_i}}{C_{i,\infty}^2\Bigl(\sqrt{\overline{C_i^2}} + \overline{C_i}\Bigr)} \leq \frac{1}{C_{i,\infty}^2}.
	\end{equation}
	Similarly to \eqref{ee1}, we use Taylor expansion and  ansatz \eqref{ansatz} to get
	\begin{equation}\label{e4}
		\begin{aligned}
			\sum_{r=1}^{|\mathcal R|}\left[\frac{\overline{\CC}^{\yy_r}}{\WW^{\yy_r}} - \frac{\overline{\CC}^{\yy_r'}}{\WW^{\yy_r'}}\right]^{\!2}
			&= \sum_{r=1}^{|\mathcal R|}\left(\prod\limits_{i=1}^{N}\left[\sqrt{\frac{\overline{c_i}}{c_{i,\infty}}} - Q(C_i)\|\delta_i\|\right]^{y_{r,i}} - \prod\limits_{i=1}^{N}\left[\sqrt{\frac{\overline{c_i}}{c_{i,\infty}}} - Q(C_i)\|\delta_i\|\right]^{y_{r,i}'}\right)^{\!2}\\
			&\geq \frac{1}{2}\sum_{r=1}^{|\mathcal R|}\left[ \sqrt{\frac{\overline{\cc}}{\ww}}^{\yy_{r}} - \sqrt{\frac{\overline{\cc}}{\ww}}^{\yy_{r}'}\right]^{\!2} - N|\mathcal R|\max_{i} \{\mathcal H_i\}\,\sum_{i=1}^{N}\|\delta_i\|^2
		\end{aligned}
	\end{equation}
	in which the constant $\mathcal H_i$ is estimated as
	\begin{equation*}\label{Hi}
		\mathcal H_i = 2\!\max_{\yy\in \mathcal C,1\le j\le N}\!\left\{\frac{y_i(\sqrt{\tilde{K}}(1+\sqrt{c_{i,\infty}}))^{y_i-1}}{c_{i,\infty}^{y_i+1}}\!\!\left[\sqrt{\frac{\tilde{K}}{c_{j,\infty}}}^{y_j}\!\! + \frac{y_j\sqrt{\tilde{K}}}{c_{j,\infty}}\!\!\left(\frac{\sqrt{\tilde{K}}(1+\sqrt{c_{i,\infty}})}{c_{i,\infty}}\right)^{\!y_j-1}\right]^{\!N-i}\right\}.
	\end{equation*}
%	where $\widetilde{Q}(\overline{c_i},R(C_i), \|\delta_i\|)$ is uniformly bounded $|\widetilde{Q}(\overline{c_i},R(C_i), \|\delta_i\|)|\leq L$ thanks to the boundedness of $\overline{c_i}$, $R(C_i)$ and $\|\delta_i\|$. 
	Hence, it follows from Lemma \ref{lem:5} and \eqref{e4} and the Poincar\'{e} inequality that for any $\theta \in (0,1)$
	\begin{equation*}
		\begin{aligned}
			\mathcal{D}(\cc)&\geq K_3\left(\sum_{i=1}^{N}C_P\|\delta_i\|^2 + \theta \sum_{r=1}^{|\mathcal R|}\biggl[\frac{\overline{\CC}^{\yy_r}}{\WW^{\yy_r}} - \frac{\overline{\CC}^{\yy_r'}}{\WW^{\yy_r'}}\biggr]^{\!2}\right)\\
			&\geq K_3\left((C_P - \theta N|\mathcal R|\max_{i}\{\mathcal H_i\})\sum_{i=1}^{N}\|\delta_i\|^2 + \frac{\theta}{2}\sum_{r=1}^{|\mathcal R|}\left[\sqrt{\frac{\overline{\cc}}{\ww}}^{\yy_{r}} - \sqrt{\frac{\overline{\cc}}{\ww}}^{\yy_{r}'}\right]^{\!2}\right)\\
			&\geq K_2\sum_{r=1}^{|\mathcal R|}\left[\sqrt{\frac{\overline{\cc}}{\ww}}^{\yy_{r}} - \sqrt{\frac{\overline{\cc}}{\ww}}^{\yy_{r}'}\right]^{\!2}
		\end{aligned}
	\end{equation*}
	with 
	\begin{equation}
	\label{K2}
	K_2 = \frac 12 K_3\min\left\{\frac 12; C_P(N|\mathcal R|\max_{i}\{\mathcal{H}_i\})^{-1} \right\}
	\end{equation}
	by choosing $\theta = \min\left\{\frac 12; C_P(N|\mathcal R|\max_{i}\{\mathcal{H}_i\})^{-1} \right\}$.
\end{proof}

The last step now is to prove the finite dimensional inequality \eqref{last_1_re}. Let us recall that until this point, we have not used the fact that the system under consideration possesses no boundary equilibria. This fact turns out to be very useful when dealing with systems having boundary equilibria (see Section \ref{sec:3}).
\begin{lemma}\label{lem:finite}
	Assuming that the chemical reaction network $(\mathcal S, \mathcal C, \mathcal R, \mathcal K)$ is complex balanced and does not have any boundary equilibria. Then, for any $\overline{\cc}\in \mathbb R^N_{>0}$ satisfying $\mathcal{E}(\overline{\cc}|\ww) <+\infty$ and $\mathbb Q\,\overline{\cc} = M$, the inequality \eqref{last_1_re} holds for some constant $H_1>0$.
\end{lemma}
\begin{remark}
	We remark here that while all the constants in previous lemmas can be explicitly estimated, the constant $H_1$ in \eqref{last_1_re} (as established in the this lemma) is in general not explicit since the proof utilises a contradiction argument. However, we believe that for any concrete system, where the conservation laws are explicitly known, $H_1$ can be computed explicitly via only elementary calculations (see Section \ref{Applications}).
	Estimating $H_1 $ for general systems is a subtle issue since the structure of conservation laws, which is crucial for an explicit estimate, is unclear  in general and remains thus an open problem.
\end{remark}
\begin{proof}
	Observe that the right hand side of \eqref{last_1_re} equals zero if and only if $\overline{\cc} = \ww$. Therefore we first prove that the left hand side of \eqref{last_1_re} can only be zero when $\overline{\cc} \equiv \ww$. Indeed, assuming that the left hand side of \eqref{last_1_re} is zero, then we have
	\begin{equation}\label{s1}
		\frac{\overline{\cc}^{\yy_r}}{\ww^{\yy_r}} = \frac{\overline{\cc}^{\yy_r'}}{\ww^{\yy_r'}} \quad \text{ which implies } \quad \frac{\overline{\cc}^{\yy_r'}}{\ww^{\yy_r'}}\ww^{\yy_r} = \overline{\cc}^{\yy_r}, \qquad \text{for all } r = 1, \ldots, |\mathcal R|.
	\end{equation}
	Thus, for any $\yy \in \mathcal C$ we have
	\begin{equation*}
		\begin{aligned}
		\sum_{\{r:\,\yy_r=\yy\}}k_r\overline{\cc}^{\yy_r}  &=  \frac{\overline{\cc}^{\yy}}{\ww^{\yy}} \sum_{\{r:\,\yy_r=\yy\}}k_r\ww^{\yy_r} 	=\frac{\overline{\cc}^{\yy}}{\ww^{\yy}}\sum_{\{s:\, \yy_s' = \yy\}}k_s\ww^{\yy_s}\ &&(\text{using the condition (\ref{ComplexBalance})})\\
		&= \sum_{\{s:\,\yy_s' = \yy\}}k_s\frac{\overline{\cc}^{\yy_s'}}{\ww^{\yy_s'}}\ww^{\yy_s} = \sum_{\{s:\,\yy_s' = \yy\}}k_s\overline{\cc}^{\yy_s} &&(\text{using (\ref{s1})}).
		\end{aligned}
	\end{equation*}
	That means that $\overline{\cc}$ is a complex balanced equilibrium. Since the chemical reaction network has no other complex balanced equilibrium than $\ww$, we obtain the desired claim that $\overline{\cc} \equiv \ww$.
	
	Now define
	\begin{equation*}
		H_1  = \inf_{\overline{\cc}\in \Sigma_{K, M}}\frac{\sum\limits_{r=1}^{|\mathcal R|}\left[\sqrt{\frac{\overline{\cc}}{\ww}}^{\yy_r} - \sqrt{\frac{\overline{\cc}}{\ww}}^{\yy_r'}\right]^2}{\sum\limits_{i=1}^{N}\left(\sqrt{\frac{\overline{c_i}}{c_{i,\infty}}} - 1\right)^2},
\end{equation*}
%where the infimum is taken for all $\overline{\cc}\in \Sigma_{K, M}$ 
where $\Sigma_{K,M} = \{\overline{\cc}\in [0,K]^{N}: \mathbb Q\,\overline{\cc} = M\}$ and $K$ is the constant in Lemma \ref{lem:3}, i.e. in the estimate $\overline{\cc_i} \leq K$ for all $i=1,\ldots, N$, which is implied from $\mathcal{E}(\overline{\cc}|\ww) <+\infty$. Since either sides of \eqref{last_1_re} equal zero if and only if $\overline{\cc} = \ww$ and the fact that the denominator of the above fraction is bounded above, we deduce that $H_1$ can possibly only be zero if and only if $\Xi = 0$ where $\Xi$ is defined by
	\begin{equation*}
		\Xi = \liminf_{\Sigma_{K,M} \ni \overline{\cc} \rightarrow \ww}\frac{\sum\limits_{r=1}^{|\mathcal R|}\left[\sqrt{\frac{\overline{\cc}}{\ww}}^{\yy_r} - \sqrt{\frac{\overline{\cc}}{\ww}}^{\yy_r'}\right]^2}{\sum\limits_{i=1}^{N}\left(\sqrt{\frac{\overline{c_i}}{c_{i,\infty}}} - 1\right)^2}.
	\end{equation*}
	It is obvious that $\Xi \geq 0$. Now assume by contradiction that $\Xi = 0$. By linearising both the nominator and denominator around $\ww$, and by setting $\boldsymbol{\sigma} = \overline{\cc} - \ww$ and $\boldsymbol{\eta} = \frac{\boldsymbol{\sigma}}{\ww} = \left(\frac{\sigma_1}{c_{1,\infty}}, \ldots, \frac{\sigma_N}{c_{N,\infty}}\right)$, we obtain 
\begin{equation*}
		\Xi = 2\liminf_{\Sigma_{K,M}\ni \overline{\cc}\to\ww}\frac{\sum\limits_{r=1}^{|\mathcal R|}\left[\sum\limits_{i=1}^{N}\frac{y_{r,i} - y_{r,i}'}{c_{i,\infty}}(\overline{c_i} - c_{i,\infty})\right]^2}{\sum\limits_{i=1}^{N}\frac{(\overline{c_i}-c_{i,\infty})^2}{c^2_{i,\infty}}}= 2\liminf_{\Sigma_{K,M}\ni \overline{\cc}\to\ww}\frac{\sum\limits_{r=1}^{|\mathcal R|}\left[(\yy_r - \yy_r')\cdot \boldsymbol{\eta}\right]^2}{\boldsymbol{\eta}^2}.
\end{equation*}
%	By setting $\boldsymbol{\sigma} = \overline{\cc} - \ww$ and $\boldsymbol{\eta} = \frac{\boldsymbol{\sigma}}{\ww} = \left(\frac{\sigma_1}{c_{1,\infty}}, \ldots, \frac{\sigma_N}{c_{N,\infty}}\right)$ we have
%	\begin{equation*}
%		\Xi = 
%	\end{equation*}
Note that $\boldsymbol{\eta}$ is the same vector for all 
$r=1,\ldots,|\mathcal R|$ in the numerator. Note moreover
that both numerator and denominator %of the last fraction 
are of homogeneity two. We can thus rescale and normalise $\boldsymbol{\eta}$ w.l.o.g. and only consider $\boldsymbol{\eta}$ on the unit ball, that is $|\boldsymbol{\eta}| = 1$. Moreover, $\Xi = 0$ if and only if the nominator is zero: 
%that is
	\begin{equation*}
		\sum_{r=1}^{|\mathcal R|}\left[(\yy_r - \yy_r')\cdot \boldsymbol{\eta}\right]^2 = 0
	\end{equation*}
	which is only possible when $\boldsymbol{\eta}\in {\mathrm{ker}}(W)$, where we recall that $W$ is the Wegscheider matrix
	\begin{equation*}
		W = [(\yy_r'- \yy_r)_{r=1,\ldots,|\mathcal R|}]^{\top}\in \mathbb R^{|\mathcal R|\times N}.
\end{equation*}

Recall that $m = \mathrm{codim}(W)= \dim(\ker(W))$ is the number of conservation laws.
If $m = 0$ and the system \eqref{e0}--\eqref{Reaction} does not have a conservation law and equivalently $\ker(W) = \{0\}$, then it follows that $\boldsymbol{\eta} = 0$, which is a contradiction to $|\boldsymbol{\eta}| = 1$.
If $m>0$, then by using $\boldsymbol{\eta}\in \mathrm{ker}(W)$ and the fact that the rows of $\mathbb Q$ form a basis of $\mathrm{ker}(W)$, it follows that $\boldsymbol{\eta} = \mathbb Q^{\top}\boldsymbol{\gamma}$ with some $\boldsymbol{\gamma}\in \mathbb R^m$. Since $\mathbb Q\,\boldsymbol{\sigma} = \mathbb Q\,(\overline{\cc} - \ww) = M - M = 0$, we obtain (by recalling $\boldsymbol{\eta} = \frac{\boldsymbol{\sigma}}{\ww}$)
	\begin{equation*}
		0 =  \mathbb Q\,\boldsymbol{\sigma} = \mathbb Q\,\mathrm{diag}(\ww)\,\boldsymbol{\eta} = \mathbb Q\,\mathrm{diag}(\ww)\,\mathbb Q^{\top}\boldsymbol{\gamma}
	\end{equation*}
	which implies $\boldsymbol{\gamma} = 0$ since $\mathbb Q$ has full rank. Thus $\boldsymbol{\eta} = 0$ which again contradicts with $|\boldsymbol{\eta}| = 1$.
	
	In conclusion, we have proved that $\Xi>0$, which implies the existence of a constant $H_1 >0$ and hence completes the proof.
%	%\todo[inline]{Some arguments}
%	\textcolor{red}{Some arguments!}
%	The contradiction shows that $\Xi>0$ which leads to $H_1 >0$.
\end{proof}

%\begin{lemma}[(Upper bound for concentrations)]\label{lem:bounded}
%	There exists a constant $\tilde{K}>0$ such that
%	\begin{equation*}
%	\overline{c_i}(t) \leq \tilde{K} \quad \text{ for all } t>0.
%	\end{equation*}
%\end{lemma}
%\begin{proof}
%	Thanks to the estimate \eqref{boundentropy_approx} and the fact that $\cc^{\varepsilon} \to \cc$ and $\cc^{\varepsilon}_0 \to \cc_0$ almost everywhere, we obtain
%	\begin{equation*}
%	\mathcal{E}(\cc(t)|\ww) \leq \mathcal{E}(\cc_0|\ww).
%	\end{equation*}
%	Now by the elementary inequalities $x\log(x/y) - x + y \geq (\sqrt{x} - \sqrt{y})^2 \geq x/2 - y$ for $x,y\geq 0$, we have
%	\begin{equation*}
%	\mathcal{E}(\cc_0|\ww) \geq \sum_{i=1}^{N}\int_{\Omega}(\sqrt{c_i} - \sqrt{c_{i,\infty}})^2dx \geq \frac 12 \sum_{i=1}^{N}\overline{c_i}(t) - \sum_{i=1}^{N}c_{i,\infty}.
%	\end{equation*}
%	Hence we have proved Lemma \ref{lem:bounded} with 
%	\begin{equation*}
%	\tilde{K}:= 2\left(\mathcal{E}(\cc_0|\ww) + \sum_{i=1}^{N}c_{i,\infty}\right).
%	\end{equation*}
%\end{proof}

We can now begin the
\begin{proof}[Proof of Theorem \ref{theo:main}]
	From Lemmas \ref{lem:3}, \ref{lem:4} and \ref{lem:finite} we get
	\begin{equation*}
		\mathcal{D}(\cc) \geq \frac{K_2H_1}{K_1}\mathcal{E}(\overline{\cc}|\ww),
	\end{equation*}
	which in combination with Lemma \ref{lem:2} leads to the desired estimate \eqref{MainEstimate}. 
	
Next, thanks to Theorem \ref{renormalised}, any renormalised solution satisfies the conservation laws and the weak entropy-entropy dissipation law \eqref{weak}. Hence we can apply a variant version of Gronwall's inequality (see e.g. \cite{FL16} or \cite{Wil}) to get the exponential decay
	\begin{equation*}
		\mathcal{E}(\cc(t)|\ww) \leq e^{-\lambda t}\mathcal{E}(\cc_0|\ww)
	\end{equation*}
	for almost all $t>0$. This convergence in a combination with the Csisz\'{a}r-Kullback-Pinsker type inequality in Lemma \ref{lem:CKP} leads to the claimed convergence to equilibrium \eqref{convergence}.
\end{proof}
%\begin{remark}\label{change-variables}
%	In the applications to specific systems, we will prove the inequality \eqref{eq_mu} explicitly then thus obtain entirely the explicit convergence to equilibrium. 	
%\end{remark}
\subsection{Applications to reversible enzyme reactions}\label{Applications}
Theorem \ref{theo:main} shows that any renormalised solution of complex balanced reaction-diffusion systems without boundary equilibria  converges exponentially to equilibrium with a constant and a rate, which can be explicitly estimated up to the finite dimensional inequality \eqref{last_1_re}. Proving  \eqref{last_1_re} with an explicit constant $H_1$ seems to be a difficult task in full generality due to the non-convex nonlinear reaction terms and the non-explicit structure of conservation laws, i.e. due to the fact that we have no explicit structure of the constraints imposed by the matrix $\mathbb Q$ . 
	
In this section, however, we will show that for a specific system, where the conservation laws are explicitly known, we can prove inequality \eqref{last_1_re} with an explicit constant $H_1$ by using elementary estimates. Hence we obtain convergence to equilibrium for \eqref{e0} with explicit bounds for the convergence rates and constants in a highly relevant model of enzyme reactions.
\medskip
	
	For notational convenience, we use a change of variables and rewrite the finite dimensional inequality \eqref{last_1_re} in a form, which is easier to handle in the specific case at hand.
	By denoting
	\begin{equation}\label{perturbation}
		\overline{c_i} = c_{i,\infty}(1+\mu_i)^2 \quad \text{ or equivalently } \quad \overline{\cc} = \ww(1 + \mm)^2
	\end{equation}
	for $\mu_i\in [-1,+\infty)$ and $\mm = (\mu_1, \ldots, \mu_N)$, inequality \eqref{last_1_re} rewrites as follow:
	\begin{equation}\label{eq_mu}
		\sum_{r=1}^{|\mathcal R|}\left[(1+\mm)^{\yy_r} - (1+\mm)^{\yy_r'}\right]^2\geq H_1\sum_{i=1}^{N}\mu_i^2,
	\end{equation}
where $\mm$ satisfies the following constraint inherited from the mass conservation laws $\mathbb Q\,\overline{\cc} = M = \mathbb Q\,\ww$
	\begin{equation}\label{constraint_mu}
		\mathbb Q\,\ww(\mm^2 + 2\mm) = 0,
	\end{equation}
and where we recall the convention $\ww(\mm^2 + 2\mm) = (c_{i,\infty}(\mu_i^2 + 2\mu_i))_{i=1,\ldots, N}$. 		

	\medskip	
{We apply our approach to a reversible variant of the famous Michaelis-Menten enzyme reaction}
	\begin{equation}\label{enzyme-reaction}\tag{E}
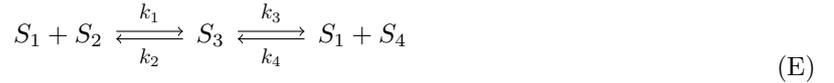
		\begin{tikzpicture}
			\node (a) {$S_1+S_2$} node (b)  at (2,0) {$S_3$} node(c) at (4,0) {$S_1+S_4$};
				\draw[arrows=->] ([xshift =0.5mm, yshift=.5mm]a.east) -- node [above] {\scalebox{.8}[.8]{$k_1$}} ([xshift =-0.5mm, yshift=.5mm]b.west);
				\draw[arrows=->] ([xshift =-0.5mm,yshift=-.5mm]b.west) -- node [below] {\scalebox{.8}[.8]{$k_2$}} ([xshift =0.5mm,yshift=-.5mm]a.east);
				\draw[arrows=->] ([xshift =0.5mm, yshift=.5mm]b.east) -- node [above] {\scalebox{.8}[.8]{$k_3$}} ([xshift =-0.5mm, yshift=.5mm]c.west);
				\draw[arrows=->] ([xshift =-0.5mm,yshift=-.5mm]c.west) -- node [below] {\scalebox{.8}[.8]{$k_4$}} ([xshift =0.5mm,yshift=-.5mm]b.east);
		\end{tikzpicture}
	\end{equation}
	For the sake of clarity, we shall assume $k_1=k_2=k_3=k_4=1$, but we emphasise that the subsequent analysis can be equally carried out for general $k_i >0$, $i=1,\ldots, 4$ without additional technical difficulties. The corresponding mass action reaction-diffusion system reads as
	\begin{equation}\label{enzyme}
		\begin{cases}
			\partial_t c_1 - d_1\Delta c_1 = - c_1c_2 - c_1c_4 + 2c_3,&\quad x\in \Omega, \quad t>0,\\
			\partial_t c_2 - d_2\Delta c_2 = - c_1c_2 + c_3,&\quad x\in \Omega, \quad t>0,\\
			\partial_t c_3 - d_3\Delta c_3 = c_1c_2 + c_1c_4 - 2c_3,&\quad x\in \Omega, \quad t>0,\\
			\partial_t c_4 - d_4\Delta c_4 = -c_1c_4 + c_3, &\quad x\in \Omega, \quad t>0,
		\end{cases}
	\end{equation}
	with homogeneous Neumann boundary conditions $\nabla c_i\cdot \nu = 0$ on $\partial\Omega$ and non-negative initial data $c_i(x,0) = c_{i,0}(x)\ge0$, $i=1,\ldots, 4$, in which $\Omega$ is a bounded domain with sufficiently smooth boundary  (e.g. $\partial\Omega\in C^{2+\epsilon}$ with $\epsilon>0$) and normalised volume $|\Omega|=1$. The large time behaviour of various reaction-diffusion models of reversible enzyme kinetics has also been recently studied in e.g. \cite[Section 8]{JanThesis} or \cite{FT15}.
			
	It is easy to check that there are two linear independent mass conservation laws for \eqref{enzyme} and that the matrix $\mathbb Q$ can be chosen as
	\begin{equation*}
		\mathbb Q = \begin{pmatrix}
			1&0&1&0\\
			0&1&1&1
		\end{pmatrix},
	\end{equation*}
	which implies the conservation laws
	\begin{equation}\label{enzyme-conservation}
		\overline{c_1}(t) + \overline{c_3}(t) = \overline{c_{1,0}} + \overline{c_{3,0}} =: M_{13} \quad \text{ and } \quad \overline{c_2}(t) + \overline{c_3}(t) + \overline{c_4}(t) = \overline{c_{2,0}} + \overline{c_{3,0}} + \overline{c_{4,0}} =: M_{234}.
	\end{equation}
	Once the positive initial masses $M_{13}>0$ and $M_{234}>0$ are fixed, then the unique positive equilibrium $(c_{1,\infty}, c_{2,\infty}, c_{3,\infty}, c_{4,\infty})$ to \eqref{enzyme} is determined by
	\begin{equation}\label{enzyme-equi}
		\begin{cases}
			c_{1,\infty}c_{2,\infty} = c_{3,\infty} = c_{1,\infty}c_{4,\infty},\\						
			c_{1,\infty} + c_{3,\infty} = M_{13}\\
			c_{2,\infty} + c_{3,\infty} + c_{4,\infty} = M_{234}.
		\end{cases}
	\end{equation}
It is straightforward to check that this equilibrium is a complex balanced equilibrium (and even a detailed balanced equilibrium) and that system \eqref{enzyme} possesses no boundary equilibria. The existence of global renormalised solution to \eqref{enzyme} follows immediately from Theorem \ref{renormalised}. Moreover, since the nonlinearities in \eqref{enzyme} are quadratic, it is well-known (see e.g. \cite{Pie,Pie10}) that \eqref{enzyme} has a global weak solution. Moreover, thanks to the special structure of \eqref{enzyme}, we show in the following that these weak solutions are in fact strong solutions and grow at most polynomially in time.
	\begin{proposition}\label{regular-solution}
		Let $\Omega\subset \mathbb R^n$ be a bounded domain with smooth boundary $\partial\Omega$ (e.g. $C^{2+\epsilon}$ with $\epsilon>0$). Assume that the initial data $\cc_0 = (c_{1,0}, \ldots, c_{4,0}) \in L^{\infty}(\Omega)^4$, then any weak solution $\cc(x,t)$ to \eqref{enzyme} belongs to $L^{\infty}_{loc}(0,\infty;L^{\infty}(\Omega))^4$ and moreover
		\begin{equation*}
			\|c_{i}(t)\|_{L^{\infty}(0,T;L^{\infty}(\Omega))} \leq C_T \quad \text{ for all } 0\le t\le T, \quad i=1,\ldots, 4,
		\end{equation*}
where $C_T$ is a constant depending {\normalfont polynomially} on $T$; i.e. there exists a polynomial $P(T)$ such that $C_T \leq P(T)$ for all $T>0$.
\end{proposition}
\begin{remark}
Note that the $L^{\infty}$-bounds of Proposition \ref{regular-solution} are sufficient to apply standard parabolic bootstrap arguments and show that $\cc(x,t)$ is indeed a classical solution (or even smooth if $\partial\Omega$ is smooth) and thus unique. 
\end{remark}	
\begin{proof}[Proof of Proposition \ref{regular-solution}]
		The proof relies on duality estimates and comparison  principle arguments for scalar parabolic equations, which exploit the special structure of  \eqref{enzyme}. In this proof we always denote by $C_T$ a general constant depending polynomially on $T>0$. First, it follows from \eqref{enzyme} that
		\begin{equation*}
			\begin{gathered}
				\partial_t(c_1 + c_3) - \Delta(d_1c_1 + d_3c_3) = 0,\\
				\partial_t(c_2+c_3+c_4) - \Delta(d_2c_2 + d_3c_3 + d_4c_4) = 0.
			\end{gathered}
		\end{equation*}
By a classical duality estimate (see e.g. \cite{PS00}) and by denoting $L^2(Q_T) = L^2(0,T;L^2(\Omega))$, we have
\begin{equation*}
\|c_i\|_{L^2(Q_T)} \leq C_T, \qquad \text{ for all } i=1,\ldots, 4.
\end{equation*}
{Moreover, \eqref{enzyme} is quasi-positive in the sense of e.g. 
\cite{Pie10} and thus preserves non-negativity of weak solutions $c_1, \ldots, c_4$ from non-negative initial data.} This, implies
\begin{equation}\label{bootstrap}
\partial_tc_1 - d_1c_1 \leq 2c_3, \qquad \partial_tc_2 - d_2\Delta c_2 \leq c_3, \qquad \partial_tc_4 - d_4\Delta c_4 \leq c_3.
\end{equation}
{Next, we by recalling \cite[Lemma 3.3]{CDF14}, there exists a 
constant $C_T$, which depend polynomially on $T$ and quantifies the smoothing effect of the heat operator in the following sense:} Given $f\in L^p(Q_T)$ and let $v$ be the solution to $v_t - d\Delta v = f$ subject to homogeneous Neumann boundary condition. Then, 
		\begin{itemize}
			\item if $p < (N+2)/2$ then $\|v\|_{L^{s-\epsilon}(Q_T)} \le C_T$ for any $\epsilon>0$ with $s = \frac{(N+2)p}{N+2-2p}$,
			\item if $p\geq (N+2)/2$ then $\|v\|_{L^r(Q_T)} \leq C_T$ for all $1\leq r <+\infty$.
		\end{itemize}
		Therefore it follows from \eqref{bootstrap} and $c_3 \in L^2(Q_T)$ in particular that
		\begin{equation*}
			\|c_1\|_{L^{s-\varepsilon}(Q_T)} \leq C_T \quad \text{ with } \quad s = \frac{2(N+2)}{N-2} \quad \text{ for } \quad N\geq 3
		\end{equation*}
		and 
		\[
			\|c_1\|_{L^r(Q_T)} \leq C_T \quad \text{ for all } \quad 1 \leq r <\infty \quad \text{ if } \quad N = 1,2.
		\]
		On the other hand, by another duality estimate (see e.g. \cite[Lemma 33.3]{QS07}), it follows from
		\begin{equation*}
			\partial_t(c_1+c_3) - \Delta(d_1c_1 + d_3c_3) = 0
		\end{equation*}
that the regularity and the polynomial dependence of $C_T$ on $T$ are transferred from  $c_1$ to $c_3$, which implies that
$\|c_3\|_{L^{s-\epsilon}(Q_T)} \leq C_T$ for all $\epsilon>0$ if $N \geq 3$, and $\|c_3\|_{L^r(Q_T)} \leq C_T$ for all $r\in [1,\infty)$ if $N=1,2$. By repeating this procedure, we obtain after finitely many steps that $\|c_3\|_{L^q(Q_T)} \leq C_T$ with $q \geq \frac{N+2}{2}$. Then, \eqref{bootstrap} implies
		\begin{equation*}
			\|c_1\|_{L^r(Q_T)}, \|c_2\|_{L^r(Q_T)}, \|c_4\|_{L^r(Q_T)} \leq C_T
		\end{equation*}
		for all $r\in [1,\infty)$, which yields in return $\|c_3\|_{L^r(Q_T)} \leq C_T$ for all $r\in [1,\infty)$. Hence, after one application of the classical smoothing effect of heat operator, the proof of the Proposition is completed.
	\end{proof}
	\begin{theorem}\label{theo:enzyme}	
		Assume that $\Omega$ is a bounded domain with smooth boundary (e.g. $C^{2+\epsilon}$ for $\epsilon>0$). Fix the initial masses $M_{13}>0$ and $M_{234}>0$.
		
Then, any renormalised solution $\cc = (c_1, \ldots, c_4)$ to \eqref{enzyme} subject to initial data $\cc_0 = (c_{i,0})_{i=1,\ldots, 4}$ having initial masses $M_{13}$ and $M_{234}$ and satisfying $\sum_{i=1}^{4}\int_{\Omega}c_{i,0}\log c_{i,0}dx < +\infty$, converges in $L^1$ exponentially to the unique positive equilibrium $\ww$ as defined in \eqref{enzyme-equi}: 
\begin{equation*}
			\sum_{i=1}^{4}\|c_i(t) - c_{i,\infty}\|_{L^1(\Omega)}^2 \leq Ce^{-\lambda t}, \qquad \text{for a.a. } t\ge 0,
		\end{equation*}
where the constant $C$ and the rate $\lambda$ can be explicitly estimated in terms of $\Omega$, the equilibrium $\ww$ and initial masses $M_{13}$ and $M_{234}$.
		
		Moreover, if the initial data $\cc_0$ belongs to $L^{\infty}(\Omega)^4$, then \eqref{enzyme} has a unique global classical solution, which converges exponentially to $\ww$ in any $L^{p}$-norm for $1\leq p<\infty$, i.e.
		\begin{equation}\label{Linfinity}
			\sum_{i=1}^{4}\|c_i(t) - c_{i,\infty}\|_{L^{p}(\Omega)} \leq Ce^{-\lambda' t}, \qquad \text{for all}\quad  t\ge 0,
		\end{equation}
		with explicit constant $C$ and rate $\lambda'$.
\end{theorem}
	\begin{proof}
		Since the system satisfies the complex balanced condition and possesses no boundary equilibria, Theorem \ref{theo:main} implies immediately that any renormalised solution converges exponentially to the equilibrium defined in \eqref{enzyme-equi}. It remains to bound of convergence rate explicitly. Thanks to Theorem \ref{theo:main} and \eqref{eq_mu} that means to compute explicitly a constant $H_1^{enzyme}>0$ in the finite dimensional inequality
		\begin{equation}\label{finite_enzyme}
			\begin{gathered}\hfill
			\left[(1+\mu_1)(1+\mu_2) - (1+\mu_3)\right]^2 + [(1+\mu_3) - (1+\mu_1)(1+\mu_4)]^2 \geq H_1^{enzyme}(\mu_1^2 + \mu_2^2 + \mu_3^2 + \mu_4^2)
			\end{gathered}
		\end{equation}
for all $\mu_i\in [-1,\infty)$ satisfying the following constraints, which are equivalent to the mass conservation laws \eqref{enzyme-conservation}:
		\begin{subequations}
			\begin{equation}\label{i1}
				c_{1,\infty}(\mu_1^2 + 2\mu_1) + c_{3,\infty}(\mu_3^2 + 2\mu_3) = 0,
			\end{equation}
			\begin{equation}\label{i2}
				c_{2,\infty}(\mu_2^2 + 2\mu_2) + c_{3,\infty}(\mu_3^2 + 2\mu_3) + c_{4,\infty}(\mu_4^2 + 2\mu_4) = 0.
			\end{equation}
		\end{subequations}
		Note that \eqref{finite_enzyme} is the speficic form of inequality \eqref{eq_mu} in case of the reversible enzyme reaction \eqref{enzyme-reaction}.
Let $G$ denote the left hand side of \eqref{finite_enzyme}. First, the elementary inequality $a^2+b^2\ge (a-b)^2/2$ yields 
		\begin{equation*}
			G\geq \frac 13\biggl([(1+\mu_1)(1+\mu_2) - (1+\mu_3)]^2 + [(1+\mu_3) - (1+\mu_1)(1+\mu_4)]^2 + (1+\mu_1)^2(\mu_2 - \mu_4)^2\biggr)
		\end{equation*}
		From \eqref{i1} and by observing that $(\mu_1+2), (\mu_3+2) \geq 1$, it follows directly that $\mu_1$ and $\mu_3$ must have different signs, which leads to the following two cases:
		\begin{itemize}
			\item[i)] Consider $\mu_1\geq 0$ and $\mu_3\leq 0$: \\
First, we have
\begin{equation}\label{i3}
(1+\mu_1)^2(\mu_2 - \mu_4)^2 \geq (\mu_2 - \mu_4)^2.
\end{equation}
From \eqref{i2} and $\mu_3 \leq 0$, we infer that at least either $\mu_2\geq 0$ or $\mu_4\geq 0$, which leads to two subcases:
				\begin{itemize}
					\item[ia)] Suppose $\mu_2\geq 0$. Then,
						\begin{equation}\label{i4}
							[(1+\mu_1)(1+\mu_2) - (1+\mu_3)]^2 = [\underbrace{\mu_2(1+\mu_1)}_{\geq 0} + \underbrace{\mu_1 - \mu_3}_{\geq 0}]^2 \geq (\mu_1 - \mu_3)^2 \geq \mu_1^2 + \mu_3^2
						\end{equation}					
since $\mu_1 \mu_3\le 0$. Similarly,
						\begin{equation}\label{i5}
							[(1+\mu_1)(1+\mu_2) - (1+\mu_3)]^2 \geq \mu_2^2 + \mu_3^2.
						\end{equation}
						It therefore follows from \eqref{i3}, \eqref{i4} and \eqref{i5} that
						\begin{equation}\label{H1}
							G \geq \frac 13\left(\frac 12  (\mu_1^2 + \mu_2^2) + \mu_3^2 + (\mu_2 - \mu_4)^2\right) \geq \frac{1}{18}(\mu_1^2 + \mu_2^2 + \mu_3^2 + \mu_4^2),
						\end{equation}
where we have used Young's inequality and the factor $\frac{1}{18}$ is not sharp but chosen in order to obtain the same lower bound \eqref{H1} as in the second case below. 
					
					\item[ib)] Suppose $\mu_4\geq 0$. In this case the term $[(1+\mu_3) - (1+\mu_1)(1+\mu_4)]^2$ can be estimated analog to case ia).
				\end{itemize}
						
			\item[ii)] Consider $\mu_1\leq 0$ and $\mu_3 \geq 0$:\\
			Since $\mu_3 \geq 0$, we obtain from \eqref{i2} that at least either $\mu_2 \leq 0$ or $\mu_4 \leq 0$. We can then use the same arguments to \eqref{i4} and \eqref{i5} to imply that if $\mu_2 \leq 0$ then
			\begin{equation}\label{i6}
				G \geq \frac{1}{6}(\mu_1^2 + \mu_2^2 + \mu_3^2)
			\end{equation}
			and if $\mu_4 \leq 0$ then 
			\begin{equation}\label{i7}
				G \geq \frac{1}{6}(\mu_1^2 + \mu_4^2 + \mu_3^2).
			\end{equation}			
However, because $\mu_1 \leq 0$, the inequality \eqref{i3} is not valid anymore. In order to bypass it, we need to consider two subcases concerning the closeness of $\mu_1$ to $-1$. %Let $\varepsilon >0$ is a small constant to be chosen later (see \eqref{varepsilon}).
				\begin{itemize}
					\item[iia)] When $(1+\mu_1)^2 \geq 1/2$, we estimate like \eqref{i3}
						\begin{equation*}
							(1+\mu_1)^2(\mu_2 - \mu_4)^2 \geq \frac 12(\mu_2 - \mu_4)^2, 
						\end{equation*}
						which leads in combination with \eqref{i6} or \eqref{i7} to 
						\begin{equation}\label{H2}
							G \geq \frac{1}{18}(\mu_1^2 + \mu_2^2 + \mu_3^2 + \mu_4^2).
						\end{equation}
					\item[iib)] Consider $(1+\mu_1)^2 \leq 1/2$. By using the mass conservation law \eqref{i1}, in the form
					\begin{equation*}
						c_{1,\infty}(1+\mu_1)^2 + c_{3,\infty}(1+\mu_3)^2 = c_{1,\infty} + c_{3,\infty},
					\end{equation*}
					we have
					\begin{equation}\label{i8}
						\mu_3 = -1 + \sqrt{1 + \frac{c_{1,\infty}}{c_{3,\infty}} - (1+\mu_1)^2\frac{c_{1,\infty}}{c_{3,\infty}}} \geq \sqrt{1 + \frac{c_{1,\infty}}{2c_{3,\infty}}} - 1.
					\end{equation}
On the other hand, we can estimate below the l.h.s. of the mass conservation law \eqref{i2}
					\begin{equation*}
						c_{2,\infty}(1+\mu_2)^2 + c_{3,\infty}^2(1+\mu_3)^2 + c_{4,\infty}(1+\mu_4)^2 = c_{2,\infty} + c_{3,\infty} + c_{4,\infty}
					\end{equation*}
					to get
					\begin{equation}\label{i9}
						\mu_2 \leq -1 + \sqrt{1 + \frac{c_{3,\infty} + c_{4,\infty}}{c_{2,\infty}}} \quad \text{ and } \quad  \mu_4 \leq -1 + \sqrt{1+\frac{c_{3,\infty}+c_{2,\infty}}{c_{4,\infty}}}.
					\end{equation}
					By combining \eqref{i8} and \eqref{i9} we have
					\begin{equation}\label{nu}
						\mu_3^2 \geq \nu_1 \mu_2^2 \quad \text{ with } \quad \nu_1 := \left(\sqrt{1 + \frac{c_{1,\infty}}{2c_{3,\infty}}} - 1\right)^2\left(-1 + \sqrt{1 + \frac{c_{3,\infty} + c_{4,\infty}}{c_{2,\infty}}}\right)^{-2}
					\end{equation}
					and similarly
					\begin{equation}\label{nu_2}
						\mu_3^2\geq \nu_2\mu_4^2\quad  \text{ with }\quad \nu_2 :=\left(\sqrt{1 + \frac{c_{1,\infty}}{2c_{3,\infty}}} - 1\right)^2\left(-1 + \sqrt{1 + \frac{c_{3,\infty} + c_{2,\infty}}{c_{4,\infty}}}\right)^{-2}.
					\end{equation}
					Thus, combining these estimates with \eqref{i6} or \eqref{i7} leads to
					\begin{equation}\label{H3}
						G\geq \min\left\{\frac{1}{18}; \frac{\nu_1}{9}; \frac{\nu_2}{9} \right\}(\mu_1^2 + \mu_2^2 + \mu_3^2 + \mu_4^2).
					\end{equation}
				\end{itemize}
		\end{itemize}
		
		In conclusion, it follows from \eqref{H1}, \eqref{H2} and \eqref{H3} that
		\begin{equation*}
			G \geq \min\left\{\frac{1}{18}; \frac{\nu_1}{9}; \frac{\nu_2}{9} \right\}(\mu_1^2 + \mu_2^2 + \mu_3^2 + \mu_4^2)
		\end{equation*}	
		with $\nu_1$ and $\nu_2$ in \eqref{nu} and \eqref{nu_2}, respectively, which proves \eqref{finite_enzyme} with 
		$$H_1^{enzyme}= \min\left\{\frac{1}{18};  \frac{\nu_1}{9}; \frac{\nu_2}{9} \right\}$$
		and hence completes the proof of explicit convergence of renormalised solutions to equilibrium for \eqref{enzyme}.
		
{Concerning the $L^p$ convergence  \eqref{Linfinity}, we interpolate Proposition \ref{regular-solution} and have for $\theta \in (0,1)$ and $p=\frac{1}{\theta}$}
		\begin{equation*}
			\|u_i(T) - u_{i,\infty}\|_{L^p(\Omega)} \leq \|u_i(T) - u_{i,\infty}\|_{L^\infty(\Omega)}^{1 - \theta}\|u_i(T) - u_{i,\infty}\|_{L^1(\Omega)}^{\theta} \leq C_T^{1-\theta}Ce^{-\lambda \theta t/2}\leq C e^{-\lambda' t},
\end{equation*}
for a constant $C$ and any $\lambda' < \frac{\lambda \theta}{2}$.
\end{proof}
	
\section{Proof of Theorem \ref{theo:boundary} and applications to systems with boundary equilibria}\label{sec:3}
%	We know that for systems possessing boundary equilibria, the entropy-entropy dissipation estimate \eqref{MainEstimate} does not hold in general, unless we take into account properties of solutions to the systems. With the help of Theorem \ref{theo:main}, and especially thanks to the linear independence of $\lambda$ on $H_1$, once we get a variant version of \eqref{last_1_re} we can pass it "linearly" to \eqref{MainEstimate} and hence obtain some convergence for the system. 
	
	\subsection{Proof of Theorem \ref{theo:boundary}}
	\begin{proof}[Proof of Theorem \ref{theo:boundary}]
We already mentioned in the proof of Lemma \ref{lem:finite}, that the validity of the Lemmas \ref{lem:1}, \ref{lem:2}, \ref{lem:3} and \ref{lem:4} is independent of the presence or absence of boundary equilibria. We recall here the key estimates of the Lemmas for the sake of readability: The additivity of the relative entropy allows to control the term $\mathcal{E}(\cc(t)|\overline{\cc}(t))$ via the Logarithmic Sobolev inequality in terms of the entropy dissipation, i.e.  
\begin{equation*}
			\mathcal{E}(\cc(t)|\ww) = \mathcal{E}(\cc(t)|\overline{\cc}(t)) + \mathcal{E}(\overline{\cc}(t)|\ww) \qquad\text{and}\qquad \lambda_1 \mathcal{E}(\cc(t)|\overline{\cc}(t))\leq  \mathcal{D}(\cc(t)).
		\end{equation*}
The second term $ \mathcal{E}(\overline{\cc}(t)|\ww)$ satisfies the upper bound
\begin{equation*}
\mathcal{E}(\overline{\cc}(t)|\ww) \leq K_1\sum_{i=1}^{N}\left(\sqrt{\frac{\overline{c_i}(t)}{c_{i,\infty}}} - 1\right)^2,
\end{equation*}
while the entropy dissipation obeys the lower bound 
\begin{equation*}
\mathcal{D}(\cc(t)) \geq K_2\sum_{r=1}^{|\mathcal R|}\left[\sqrt{\frac{\overline{\cc}(t)}{\ww}}^{\yy_r} - \sqrt{\frac{\overline{\cc}(t)}{\ww}}^{\yy_r'}\right]^2.
\end{equation*}		
These two estimates are connected by assumption \eqref{eq:finite_time} and we obtain all together 
		\begin{equation*}
			\mathcal{D}(\cc(t)) \geq \lambda(t)\mathcal{E}(\cc(t)|\ww)
		\end{equation*}
		with
		\begin{equation*}
			\lambda(t) = \frac 12\min\left\{\lambda_1; \frac{K_2H_1(t)}{K_1}\right\}.
		\end{equation*}
		Note that $\int_0^{+\infty}\lambda(s)ds = +\infty$ since $\int_0^{+\infty}H_1(s)ds = +\infty$ and $\lambda_1>0$. Moreover, it follows from the weak entropy-entropy dissipation law \eqref{weak} and Gronwall's inequality that
		\begin{equation*}
			\mathcal{E}(\cc(t)|\ww) \leq \mathcal{E}(\cc_0|\ww)\,e^{-\int_0^t\lambda(s)ds} \longrightarrow 0 \quad\text{ as } \quad t\to+\infty.
		\end{equation*}
		Thus the trajectory $\cc(t)$ converges to $\ww$ in relative entropy and, consequently, in $L^1$-norm due to the Csisz\'ar-Kullback-Pinsker type inequality in Lemma \ref{lem:CKP}. Therefore, after some finite time $T>0$, the solution trajectory will always stays outside of any small enough neighbourhood of all boundary equilibria. It then follows from \cite[Remark 3.6]{DFT16} that the solution converges exponentially to the positive complex balanced equilibrium.
		\end{proof}

	\subsection{Application to a specific system possessing boundary equilibria}\label{sec:system_boundary}\hfil
	
In order to show convergence to equilibrium for renormalised solutions $\cc(x,t)$ of complex balanced reaction-diffusion systems with boundary equilibria, we have to verify \eqref{eq:finite_time} as stated in Theorem \ref{theo:boundary}. 

	Similarly to Subsection \ref{Applications}, it will be convenient to change variables in the finite dimensional inequality \eqref{eq:finite_time}. By setting 
				\begin{equation}\label{change1}
					\overline{c_i}(t) = c_{i,\infty}(1+\mu_i(t))^2, \qquad \text{ for } i=1,\ldots, N,
				\end{equation}
				inequality \eqref{eq:finite_time} becomes
				\begin{equation}\label{change2}
					\sum_{r=1}^{|\mathcal R|}[(1+\mm(t))^{\yy_r} - (1+\mm(t))^{\yy_r'}]^2 \geq H_1(t)\sum_{i=1}^{N}\mu_i(t)^2
				\end{equation}
				where $\mm(t) = (\mu_1(t), \ldots, \mu_N(t))$ and the function $H_1(t)$ is required to satisfy $\int_0^{+\infty}H_1(t)dt = +\infty$. 
				
Proving \eqref{change2} for general complex balanced systems would yield a proof of the Global Attractor Conjecture (GAC), which 
is a very interesting, yet challenging open problem. Our aim in this section is to study a typical class of complex balanced systems with boundary equilibria, in which proving \eqref{change2} for renormalised solutions is a possible approach to answer the GAC in the associated PDE setting. More precisely, we consider here the reaction-diffusion systems modelling the following reaction network
	\begin{equation}\label{cycle-reaction}\tag{C}
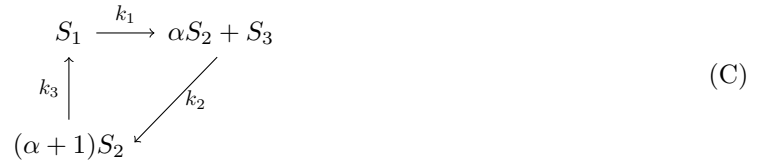
	\begin{tikzpicture} [baseline=(current  bounding  box.center)]
	\node (a) {$S_1$} node (b) at (2,0) {$\alpha S_2+S_3$} node (c) at (0,-1.5) {$(\alpha+1)S_2$};
	\draw[arrows=->] ([xshift =0.5mm]a.east) -- node [above] {\scalebox{.8}[.8]{$k_1$}} ([xshift =-0.5mm]b.west);
	\draw[arrows=->] ([yshift=0.5mm]c.north) -- node [left] {\scalebox{.8}[.8]{$k_3$}} ([yshift=-0.5mm]a.south);
	\draw[arrows=->] ([xshift =-0.5mm,yshift=-0.5mm]b.south) -- node [right] {\scalebox{.8}[.8]{$k_2$}} ([yshift =0.5mm]c.east);
	\end{tikzpicture}
	\end{equation}
	with arbitrary $\alpha \geq 1$ and $k_1, k_2, k_3 >0$. The special case $\alpha = 1$ was investigated in \cite{DFT16}. Here, 
we study the entire range 	$\alpha \geq 1$ in order to show the robustness of our arguments. 
	
The above network \eqref{cycle-reaction} is considered in a bounded domain $\Omega\subset \mathbb R^n$ with smooth boundary $\partial\Omega$ (e.g. $C^{2+\epsilon}$ for any $\epsilon>0$) and normalised volume, i.e. $|\Omega| = 1$.
	The corresponding mass action reaction-diffusion system reads as
	\begin{equation}\label{3x3}
		\begin{cases}
			\partial_tc_1 - d_1\Delta c_1 = -k_1c_1 + k_3c_2^{\alpha+1}, &\quad x\in\Omega, \quad t>0,\\
			\partial_tc_2 - d_2\Delta c_2 = k_1\alpha c_1 + k_2c_2^{\alpha}c_3 - k_3(\alpha+1)c_2^{\alpha+1},&\quad x\in\Omega, \quad t>0,\\
			\partial_tc_3 - d_3\Delta c_3 = k_1c_1 - k_2c_2^{\alpha}c_3, &\quad x\in\Omega, \quad t>0,
		\end{cases}
	\end{equation}
	subject to homogeneous Neumann boundary conditions $\nabla c_i\cdot \nu = 0$ and non-negative initial data $c_{i}(x,0) = c_{i,0}(x)$. This system has one conservation of mass, namely
	\begin{equation}\label{conservation:3x3}
		(\alpha+1)\,\overline{c_1}(t) + \overline{c_2}(t) + \overline{c_3}(t) = (\alpha+1)\, \overline{c_{1,0}} + \overline{c_{2,0}} + \overline{c_{3,0}} =: M, \qquad \text{ for all } \quad t>0.
	\end{equation}

For fixed $M>0$, system \eqref{3x3} features the boundary equilibrium $\cc^* = (c_{1}^*, c_{2}^*, c_{3}^*) = (0,0,M)$ and the unique positive complex balanced equilibrium $\ww = (c_{1,\infty}, c_{2,\infty}, c_{3,\infty})$, where $c_{2,\infty}$ is the unique positive solution to 
	\begin{equation*}
		\frac{(\alpha+1) k_3}{k_1}c_{2,\infty}^{\alpha + 1} + \frac{k_3}{k_2}c_{2,\infty}^{\alpha} + c_{2,\infty} = M
	\end{equation*}
	and 
	\begin{equation}\label{equi:3x3}
		c_{1,\infty} = \frac{k_3}{k_1}c_{2,\infty}^{\alpha+1}, \qquad \   c_{3,\infty} =  \frac{k_3}{k_2}c_{2,\infty}^{\alpha}.
	\end{equation}
		
	Due to the presence of the boundary equilibrium, we will have to apply Theorem \ref{theo:boundary} in order to show  convergence to equilibrium for renormalised solutions to \eqref{3x3}. 
	More precisely, we need to prove the modified finite dimensional inequality \eqref{eq:finite_time} or equivalently \eqref{change2} along solution trajectories of \eqref{3x3}. 
The existence of global renormalised solutions to the complex balanced system \eqref{3x3}
follows readily from Theorem \ref{renormalised}.

However, to prove inequality \eqref{change2} along renormalised solutions, we need additional information 
about these solutions, which we are only able to show for 
specific renormalised solutions constructed via a typical  
approximation scheme as already used in \cite{Fis15}.
The following lemma shows that if such a renormalised solution to \eqref{3x3} should converge to the boundary $\partial\mathbb R_{>0}^3$, then not faster than with a specific algebraic convergence rate in terms of the parameter $\alpha\ge1$.
\begin{proposition}\label{renormalized_3x3}
{For any nonnegative initial data $\cc_0 = (c_{1,0}, c_{2,0}, c_{3,0})\in L^p(\Omega)^3$, for some $1<p\le2$, which thus  satisfies }
$$
\sum_{i=1}^{3}\int_{\Omega}c_{i,0}\log(c_{i,0})dx <\infty,
$$ 
there exists a renormalised solution to \eqref{3x3}. 

Moreover, assume $\left\|1/{c_{2,0}^{\alpha}}\right\|_{L^{\infty}(\Omega)} < +\infty$. Then, {any renormalised solution, which is constructed via the below approximative scheme \eqref{approx}}, satisfies
		\begin{equation}\label{lowerbound}
			\overline{c_2}(t) \geq h(t):= \left[\frac{1}{\|1/c_{2,0}^\alpha\|_{L^{\infty}(\Omega)}} + \alpha(\alpha+1)k_3t \right]^{-1/\alpha} \quad \text{ for almost all } \quad t>0.
		\end{equation}
	\end{proposition}
\begin{remark}
We remark that Proposition \ref{renormalized_3x3} applies to all renormalised solutions which are constructed via the  approximation scheme \eqref{approx}. The lower bound \eqref{lowerbound} for an arbitrary renormalised solution 
according to the definition in Theorem \ref{renormalised}
is not clear and remains an open problem.

We also remark that the assumed $L^p$, $1<p\le2$ initial data, which we need for technical reasons in order to apply duality estimates, are slightly more restrictive than the usual $L\log L$ initial data assumption for renormalised solutions (see Theorem \ref{renormalised}). 
Note that  for $\alpha$ sufficiently larger than one, the existence of global weak solutions to system \eqref{3x3} is unclear even with $L^2$ initial data and that renormalised solutions 
are the only known global solutions 
in order to study the large time behaviour. 
\end{remark}	

\begin{proof}[Proof of Proposition \ref{renormalized_3x3}]
		The existence of a renormalised solution follows from general result in Theorem \ref{renormalised} since system \eqref{3x3} is complex balanced. 
%We now prove the lower bound \eqref{lowerbound} for renormalised solutions, which are constructed via a typical (name-giving) approximation scheme, see \cite{Fis15}. 
Due to the weak regularity of renormalised solutions, we are forced to prove \eqref{lowerbound} for sequences of 
solutions of a typical (name-giving) approximation scheme (see \cite{Fis15}) and then pass to the limit. 
We denote the nonlinearities of \eqref{3x3} by
		\begin{align*}
			%\begin{gathered}
			R_1(\cc) &= -k_1c_1 + k_3c_2^{\alpha+1},\\  R_2(\cc)& = k_1\alpha c_1+ k_2c_2^{\alpha}c_3 - k_3(\alpha+1)c_2^{\alpha+1},\\
			R_3(\cc) &= k_1c_1 - k_2c_2^{\alpha}c_3\quad  \intertext{ and } 
			\mathbf R(\cc) &= (R_1(\cc), R_2(\cc), R_3(\cc))^{\top}.
			%\end{gathered}
		\end{align*}
Moreover, denote by $|\mathbf R(\cc)| := |R_1(\cc)| + |R_2(\cc)| + |R_3(\cc)|$. Following \cite{Fis15}, we consider for $\varepsilon>0$ the approximative systems 
		\begin{equation}\label{approx}
			\partial_tc_i^{\varepsilon} - d_i\Delta c_i^{\varepsilon} = \frac{R_i(\cc^\varepsilon)}{1+\varepsilon| \mathbf R(\cc^\varepsilon)|}, \quad \nabla c_i^\varepsilon\cdot \nu = 0 \quad \text{ and } \quad 0\le c_i^\varepsilon(x,0) = c_{i,0}^\varepsilon(x)\in L^{\infty}(\Omega)
		\end{equation}
where $c_{i,0}^{\varepsilon} \to c_{i,0}$ in $L^1(\Omega)$  as $\varepsilon\to 0$. Moreover, we choose $c_{2,0}^{\varepsilon}$ such that $\|1/c_{2,0}^{\varepsilon}\|_{L^{\infty}(\Omega)} \le \|1/c_{2,0}\|_{L^{\infty}(\Omega)}$, {$c_{2,0}^{\varepsilon}\ge C\varepsilon$ for a constant $C$ and} for all $\varepsilon > 0$. {Note that for all $\varepsilon>0$, standard theory of reaction-diffusion diffusion systems implies that existence of weak, global in time solutions to \eqref{approx}.}

Moreover, by \cite{Fis15}, there exists a subsequence (not relabeled) $\{\cc^{\varepsilon} = (c_1^{\varepsilon}, c_2^{\varepsilon}, c_3^{\varepsilon})\}_{\varepsilon>0}$ such that $c_i^{\varepsilon} \to c_i$ a.e. in $\Omega\times (0,T)$, and $\cc = (c_1, c_2, c_3)$ is a global renormalised solution to \eqref{3x3}. On the other hand, we have
		\begin{equation}\label{sum-eq}
			\partial_t[(\alpha+1)c_1^{\varepsilon} + c_2^{\varepsilon} + c_3^{\varepsilon}] - \Delta[(\alpha+1)d_1c_1^{\varepsilon} + d_2c_2^{\varepsilon} + d_2c_3^{\varepsilon}] = 0.
		\end{equation}
{Hence, by duality estimates for $1<p\le2$ (see e.g. \cite{CDF14,PS00}), we have 
\begin{equation*}
\{c_i^{\varepsilon}\}_{\varepsilon>0} \text{ is bounded in } L^{p}(\Omega\times (0,T)) \text{ in terms of } \|c_{i,0}\|_{L^p(\Omega)} \text{ uniformly in } \varepsilon>0.
\end{equation*}
}
%\textcolor{red}{$L^2$ initial data for this Proposition? $c_{2,0}\in L^{\infty}$ needed?}
		
This bound combined with $c_i^{\varepsilon} \to c_i$ a.e. in $\Omega\times (0,T)$ implies $c_i^{\varepsilon} \to c_i$ in $L^1(\Omega\times (0,T))$ thanks to Vitali's theorem. Hence there exists a subsequence (not relabeled) of $c_i^{\varepsilon}$ such that for a.e. $t\in (0,T)$, $c_i^{\varepsilon}(\cdot, t) \to c_i(\cdot, t)$ a.e. in $\Omega$. Moreover, thanks to \cite{Fis15}, $\{c_i^{\varepsilon}(t)\log c_i^{\varepsilon}(t)\}_{\varepsilon>0}$ is bounded in $L^1(\Omega)$ uniformly in $\varepsilon$. The Vitali theorem implies finally that $c_i^\varepsilon(t) \to c_i(t)$ in $L^1(\Omega)$ for a.e. $t\in (0,T)$.

		\medskip		
We now prove \eqref{lowerbound} for $c_i^{\varepsilon}$. First, we remark that weak comparison arguments (see e.g. \cite{Chi00}) for the equation of $c_2^{\varepsilon}$ in \eqref{approx}, i.e. 
\begin{align*}
\partial_t c_2^{\varepsilon} - d_2\Delta c_2^{\varepsilon} = \frac{R_2(\cc^{\varepsilon})}{1+\varepsilon|\mathbf R(\cc^{\varepsilon})|} &=  \frac{1}{1+\varepsilon|\mathbf R(\cc^{\varepsilon})|}\left(k_1\alpha c_1^{\varepsilon}+ k_2(c_2^{\varepsilon})^{\alpha}c_3^{\varepsilon} - k_3(\alpha+1)(c_2^{\varepsilon})^{\alpha+1}\right)\\
&\ge \frac{- k_3(\alpha+1)(c_2^{\varepsilon})^{\alpha+1}}{1+\varepsilon|\mathbf R(\cc^{\varepsilon})|}
\ge - k_3(\alpha+1)(c_2^{\varepsilon})^{\alpha+1}
\end{align*}
subject to initial data $c_{2,0}^{\varepsilon}\ge C\varepsilon$ for a constant $C$ and for all $\varepsilon > 0$ imply the 
existence of a positive time $\tau>0$ (possibly depending in $\varepsilon$), such that $c_{2}^{\varepsilon}(x,t)\ge \frac{C\varepsilon}{2}$ 
for a.a. $x\in\Omega$ and $0\le t\le \tau$. Thus, we can test the 
equation of $c_2^{\varepsilon}$ with $-\frac{\alpha}{(c_2^{\varepsilon})^{\alpha+1}}$ as follows
\begin{equation*}
			\begin{aligned}
				\partial_t\left(\frac{1}{(c_2^{\varepsilon})^{\alpha}}\right) - d_2\Delta\left(\frac{1}{(c_2^{\varepsilon})^{\alpha}}\right)
				 &= -\frac{\alpha}{(c_2^{\varepsilon})^{\alpha+1}}(\partial_tc_2^{\varepsilon} - d_2\Delta c_2^{\varepsilon}) - \frac{d_2\alpha^2(\alpha+1)}{(c_2^{\varepsilon})^{\alpha+2}}|\nabla c_2^{\varepsilon}|^2\\
				&\leq -\frac{\alpha}{(c_2^{\varepsilon})^{\alpha+1}}\frac{1}{1+\varepsilon|\mathbf R(\cc^{\varepsilon})|}[k_1\alpha c_1^{\varepsilon}+ k_2(c_2^{\varepsilon})^{\alpha}c_3^{\varepsilon} - k_3(\alpha+1)(c_2^{\varepsilon})^{\alpha+1}]\\
				&\leq \frac{1}{1+\varepsilon|\mathbf R(\cc^{\varepsilon})|}(k_3\alpha(\alpha+1))
				\leq k_3 \alpha(\alpha+1).
			\end{aligned}
		\end{equation*}
Thus, the weak comparison principle implies again for a.a. $x\in \Omega$ and $t>0$
\begin{equation}\label{lower}
\frac{1}{[c_2^{\varepsilon}(x,t)]^{\alpha}} \leq \left\|\frac{1}{(c_{2,0}^{\varepsilon})^{\alpha}}\right\|_{L^{\infty}(\Omega)} + k_3\alpha(\alpha+1)t \leq \left\|\frac{1}{c_{2,0}^{\alpha}}\right\|_{L^{\infty}(\Omega)} + k_3\alpha(\alpha+1)t,
\end{equation}
which implies that testing with $-\frac{\alpha}{(c_2^{\varepsilon})^{\alpha+1}}$ is justified for a.a. $t\ge0$
and the lower bound \eqref{lower} holds indeed globally in time
and independently from $\varepsilon$. 

Hence $\overline{c_2^{\varepsilon}}(t)\geq h(t)$ with $h(t)$ as defined in \eqref{lowerbound}. Finally, the estimate \eqref{lowerbound} follows from \eqref{lower} and the fact that $c_2^{\varepsilon}(t) \to c_2(t)$ in $L^1(\Omega)$.
	\end{proof}
	\begin{remark}
		If the diffusion coefficients $d_1, d_2, d_3$ are close to each other, for instance, in the sense that $\delta = \max\{d_1,d_2,d_3\} - \min\{d_1,d_2,d_3\}$ is sufficiently small, then any renormalised solution to \eqref{3x3} is in fact a strong solution, see \cite{CDF14}. In these cases the arguments in Proposition \ref{renormalized_3x3} {are justified by classical maximum principle arguments as done in \cite{DFT16}. The benefit of Proposition \ref{renormalized_3x3} is to prove estimate \eqref{lowerbound} for suitable renormalised solutions without any assumption on the diffusion coefficients. This is due to the fact that \eqref{lowerbound} involves only the $L^1$-norm of $c_2$, which is preserved when passing to the limit in approximating 
renormalised solutions.}
	\end{remark}
	
	We are now ready to prove the main result of this section.
	\begin{theorem}\label{theo:3x3}
		Let $\Omega\subset\mathbb R^n$ be a bounded domain with smooth boundary $\partial\Omega$ (e.g. $C^{2+\epsilon}$ with $\epsilon>0$). Assume for system \eqref{3x3} that $\alpha \geq 1$ and $k_1, k_2, k_3>0$. 
		
Then, for any fixed positive initial mass $M>0$ and non-negative initial data {$(c_{1,0}, c_{2,0}, c_{3,0}) \in  L^{p}(\Omega)^3$ for some $1< p\le2$} having initial mass $M$ as defined in \eqref{conservation:3x3} and satisfying 
		\begin{equation}\label{assump:3x3}
			\left\|\frac{1}{c_{2,0}^{\alpha}}\right\|_{L^{\infty}(\Omega)} < +\infty,
		\end{equation}
any global renormalised solution $(c_1, c_2, c_3)$ as constructed in Proposition \ref{renormalized_3x3} converges exponentially in $L^1$ to the complex balanced equilibrium $(c_{1,\infty}, c_{2,\infty}, c_{3,\infty})$ as defined in \eqref{equi:3x3}, i.e.
\begin{equation*}
	\sum_{i=1}^{3}\|c_{i}(t) - c_{i,\infty}\|_{L^1(\Omega)}^2 \leq Ce^{-\lambda t},
\end{equation*}
for almost all $t>0$ where $C$ and $\lambda$ are constants depending {\normalfont{explicitly}} on the domain $\Omega$, the constants $\alpha, k_1, k_2, k_3$, the initial mass $M$ and $\|1/c_{2,0}^{\alpha}\|_{L^{\infty}(\Omega)}$.
	\end{theorem}
	\begin{remark}
The exponential convergence to equilibrium in Theorem \ref{theo:3x3} applies to any renormalised solution constructed via the approximation scheme of Proposition \ref{renormalized_3x3}. 
Note that it is unknown if any renormalised solution according to the definition in Theorem \ref{renormalised} can be approximated via \eqref{approx}. Thus,  
the convergence of any renormalised solution is open for future investigation.
\end{remark}
\begin{proof}
We consider renormalised solutions as constructed in Proposition \ref{renormalized_3x3}. Thanks to Theorem \ref{theo:boundary} and \eqref{change1}--\eqref{change2}, we have to find a function in time $H_1(t)$ satisfying $\int_0^{+\infty}H_1(t)dt = +\infty$ such that the following finite dimensional inequality holds
		\begin{equation}\label{final_time}
			\begin{gathered}\hfill
			[(1+\mu_1(t)) - (1+\mu_2(t))^{\alpha}(1+\mu_3(t))]^2 + [(1+\mu_2(t))^{\alpha}(1+\mu_3(t)) - (1+\mu_2(t))^{\alpha+1}]^2\\
			+ [(1+\mu_2(t))^{\alpha+1} - (1+\mu_1(t))]^2 \geq H_1(t)[\mu_1(t)^2 + \mu_2(t)^2 + \mu_3(t)^2]
			\end{gathered}
		\end{equation}
		where $\mu_i(t) \in [-1,\infty)$ is defined through $\overline{c_i}(t) = c_{i,\infty}(1+\mu_i(t))^{2}$. Note that \eqref{final_time} is the specific version of inequality \eqref{change2} for the considered reaction network \eqref{cycle-reaction}.
First, thanks to Proposition \ref{renormalized_3x3}, we have
		\[
			(1+\mu_2(t))^2 = \frac{\overline{c_2}(t)}{c_{2,\infty}} \geq \frac{h(t)}{c_{2,\infty}}
		\]
and the elementary inequality $a^2+b^2\ge (a-b)^2/2$ implies 
that the first and the third term of the left hand side of \eqref{final_time} are bounded below by 
$(1+\mu_2(t))^2 [\mu_3(t)-\mu_2(t)]^2$.
Thus, 
\begin{equation}\label{left_hand_side}
			\begin{aligned}
			\text{LHS of (\ref{final_time})}
			&\geq \min\left\{1; \left[\frac{h(t)}{c_{2,\infty}}\right] ^{\alpha}\right\}\biggl([(1+\mu_1(t)) - (1+\mu_2(t))^{\alpha}(1+\mu_3(t))]^2\\
			&\hspace{3.65cm} + [\mu_3(t) - \mu_2(t)]^2 +  [(1+\mu_2(t))^{\alpha+1} - (1+\mu_1(t))]^2\biggr).
			\end{aligned}		
		\end{equation}
Note that the quantities $\mu_i(t)$ satisfies the conservation law \eqref{conservation:3x3} in the form 
\begin{equation*}
(\alpha+1)c_{1,\infty}(\mu_1(t)^2 + 2\mu_1(t)) + c_{2,\infty}(\mu_2(t)^2 + 2\mu_2(t)) + c_{3,\infty}(\mu_3(t)^2 + 2\mu_3(t)) = 0.
\end{equation*}
By applying the Lemma \ref{finite_boundary} below, we have
\begin{multline*}	
[(1+\mu_1(t)) - (1+\mu_2(t))^{\alpha}(1+\mu_3(t))]^2 + [\mu_3(t) - \mu_2(t)]^2
+  [(1+\mu_2(t))^{\alpha+1} - (1+\mu_1(t))]^2 \\
\geq \varrho [\mu_1(t)^2 + \mu_2(t)^2 + \mu_3(t)^2],
\end{multline*}
where the constant $\varrho>0$ is defined in Lemma \ref{finite_boundary}. 

Hence, we obtain \eqref{final_time} from \eqref{left_hand_side} that
\[
H_1(t) =\varrho\min\left\{1; \left[\frac{h(t)}{c_{2,\infty}}\right] ^{\alpha}\right\} = \varrho\min\left\{1; \frac{1}{c_{2,\infty}^{\alpha}}\left[\frac{1}{\|1/c_{2,0}^\alpha\|_{L^{\infty}(\Omega)}} + \alpha(\alpha+1)k_3t \right]^{-1}\right\}
\]
Finally, it is clear that $\int_0^{\infty}H_1(t)dt = +\infty$ {since $h(t)^{-\alpha}$ is a linear function in time for all $\alpha\ge1$. }
	\end{proof}
	It remains to show
\begin{lemma}\label{finite_boundary}
Let $\alpha \geq 1$ and $a, b, c\in [-1,\infty)$ be constants satisfying 
\begin{equation}\label{constraint}
(\alpha+1)c_{1,\infty}(a^2 + 2a) + c_{2,\infty}(b^2 + 2b) + c_{3,\infty}(c^2 + 3c) = 0.
\end{equation}

Then, the following inequality holds 
\begin{equation}\label{elementary}
[(1+a) - (1+b)^{\alpha}(1+c)]^2 + [c - b]^2 + [(1+b)^{\alpha+1} - (1+a)]^2 \geq \varrho\left(a^2 + b^2 + c^2\right)
\end{equation}
with $\varrho = \min\{1/4; 1/(4(\alpha+1)\max\{1,b_{max}\}^{2\alpha})\}$ and $b_{max}$ defined in \eqref{bmax} 
\end{lemma}
\begin{remark}
Note that the constant $\varrho$ depends only on $\alpha$ and the equilibrium $\ww = (c_{1,\infty}, c_{2,\infty}, c_{3,\infty})$.
\end{remark}
	\begin{proof}
		The constraint \eqref{constraint} implies the following two cases concerning the signs of $a, b$ and $c$:
		\begin{itemize}
			\item[i)] Assume $a$ and $b$ have different signs. \\			
			In this case, we prove 
			\begin{equation*}
				[(1+b)^{\alpha+1} - (1+a)]^2 \geq a^2 + b^2.
			\end{equation*}
			Indeed, if $b\geq 0$ and $a\leq 0$, we have $(1+b)^{\alpha+1} - (1+ a)\geq (1+b) - (1+a) = b - a \geq 0$, thus 
			\begin{equation*}
				[(1+b)^{\alpha+1} - (1+a)]^2 \geq (b - a)^2 = b^2 - 2ab + a^2 \geq a^2 + b^2.
			\end{equation*}
			If $b\leq 0$ and $a\geq 0$, we have $(1+b)^{\alpha+1} \leq 1+b$ due to $0 \geq b\geq -1$ and thus $(1+a) - (1+b)^{\alpha+1} \geq (1+a) - (1+b) = a-b \geq 0$. Hence
			\[
			[(1+b)^{\alpha+1} - (1+a)]^2 = [(1+a) - (1+b)^{\alpha+1}]^2 \geq (a - b)^2 \geq a^2 + b^2.
			\]
			Hence, we estimate
			\begin{equation*}
				\text{LHS of (\ref{elementary})} \geq [c - b]^2 + a^2 + b^2 \geq a^2 +  \frac 12 b^2 +  \frac 14 c^2 \geq \text{RHS of (\ref{elementary})}
			\end{equation*}
with $\rho \le \frac{1}{4}$.
			\item[ii)] Assume $a$ and $b$ have the same sign.\\			
			In this case, \eqref{constraint} implies either ($a\geq 0$, $b\geq 0$ and $c\leq 0$) or ($a\leq 0$, $b\leq 0$ and $c\geq 0$). First, because $b$ and $c$ have different signs
			\[
				[c - b]^2 = c^2 - 2cb + b^2 \geq c^2 + b^2.
			\]
Being equivalent to \eqref{constraint}, we estimate below the l.h.s. of
			\[
				(\alpha+1)c_{1,\infty}(a+1)^2 + c_{2,\infty}(b+1)^2 + c_{3,\infty}(c+1)^2 = (\alpha+1)c_{1,\infty} + c_{2,\infty} + c_{3,\infty},
			\]
			to obtain
			\begin{equation}\label{bmax}
				b \leq -1 + \sqrt{1 + \frac{(\alpha+1)c_{1,\infty} + c_{3,\infty}}{c_{2,\infty}}} =: b_{max}.
			\end{equation}
			Next, by using Taylor's expansion, we obtain for some $\xi \in (0,b)$
			\begin{equation*}
				\begin{aligned}\hfill
				[(1+b)^{\alpha+1} - (1+a)]^2 &= [(1+(\alpha+1)\xi^{\alpha}b) - (1+ a)]^2 = [a - (\alpha+1)\xi^{\alpha}b]^2\\
				&\geq \frac 12a^2 - (\alpha+1)^2|\xi|^{2\alpha}b^2 \geq \frac 12a^2 - (\alpha+1)\max\{1,b_{max}\}^{2\alpha}b^2.
				\end{aligned}
			\end{equation*}
			Therefore,
			\begin{equation*}
				\begin{aligned}
					\text{LHS of (\ref{elementary})} &\geq [c - b]^2 + [(1+b)^{\alpha+1} - (1+a)^{2}]^2\\
					&\geq c^2 + b^2 + \min\left\{1; \frac{1}{2(\alpha+1)\max\{1,b_{max}\}^{2\alpha}} \right\}\left[\frac 12a^2 - (\alpha+1)\max\{1,b_{max}\}^{2\alpha}b^2\right]\\
					&\geq \frac 12 \min\left\{1; \frac{1}{2(\alpha+1)\max\{1,b_{max}\}^{2\alpha}} \right\}[a^2 + b^2 + c^2],
				\end{aligned}
			\end{equation*}
which proves \eqref{elementary} also in the second case and finishes the proof. 
		\end{itemize}
	\end{proof}
\noindent{\bf Aknowledgements.}\hfill

This work is partially supported by International Research Training Group IGDK 1754 and NAWI Graz.

\end{document}